\documentclass[10pt,leqno]{amsart}
\usepackage{graphicx}
\usepackage{indentfirst,csquotes}

\usepackage{amssymb,amsthm,amsmath, mathtools, mathrsfs}
\usepackage{xcolor,paralist,hyperref,fancyhdr,etoolbox, resmes}
\usepackage[T1]{fontenc}
\usepackage[english]{babel}
\usepackage[most]{tcolorbox}
\usepackage[shortlabels]{enumitem}
\usepackage{tikz}

\newtheorem{theorem}{Theorem}[]
\newtheorem{definition}[theorem]{Definition}
\newtheorem{example}[theorem]{Example}
\newtheorem{lemma}[theorem]{Lemma}
\newtheorem{proposition}[theorem]{Proposition}
\newtheorem{corollary}[theorem]{Corollary}

\newtheorem{remark}[theorem]{Remark}
\newtheorem{claim}[theorem]{Claim}

\def\L{\mathbf{L}}
\def\R{\mathbf{R}}
\def\Rkn{\mathbf{R}^{k,n}}
\def\Rnm{\mathbf{R}^{n,m}}

\def\RnmR{\mathbf{R}^{n,m} \times \R}
\def\H{\mathcal{H}}

\def\dim{\mathrm{dim}}
\def\Var{\mathrm{Var}}
\def\Div{\operatorname{Div}}
\def\spt{\operatorname{spt}}

\DeclareMathOperator*{\esssup}{ess\,sup}

\DeclareMathOperator*{\Tan}{Tan}

\newcommand{\mcal}[1]{\mathcal{#1}}
\newcommand{\abs}[1]{\left|#1\right|}
\newcommand{\norm}[1]{\left\Vert #1 \right\Vert} 
\newcommand{\parent}[1]{\left( #1 \right)} 

\newcommand{\esc}[2]{\left\langle #1, #2 \right\rangle}

\newcommand{\ol}[1]{\overline{#1}}
\newcommand{\opn}[1]{\operatorname{#1}}
\newcommand{\enorm}[1]{\abs{#1}_\mathrm{\R^{n+m}}} 
\newcommand{\lnorm}[1]{\esc{#1}{#1}}
\newcommand{\ve}{\varepsilon}

\hypersetup{ colorlinks=true, linkcolor=black, filecolor=black, urlcolor=black }

\begin{document}

\title{Spacelike Brakke Flows with Boundary in Pseudo-Euclidean Space} 
\author[J. Hervás]{Javier Hervás}
\date{\today}
\address{Departamento de Geometr\'ia y Topolog\'ia  \newline
Instituto de Matem\'aticas de Granada (IMAG) \newline
Universidad de Granada\newline
18071 Granada, Spain
}
\email{hervas@ugr.es}
\thanks{This research was partially supported by the grant PID2024-156031NB-I00 and by the  IMAG–Maria de Maeztu grant CEX2020-001105-M, both funded by MICIU/AEI/10.13039/501100011033 and ERDEF.}

\maketitle
	
\let\thefootnote\relax
	
\begin{abstract}
We develop a weak formulation of spacelike mean curvature flow in pseudo-Euclidean space. The framework is based on spacelike integer rectifiable varifolds and a pseudo-Euclidean version of first variation. Particular attention is paid to the presence of a fixed spacelike boundary. We introduce a generalized mean curvature vector and a weak spacelike conormal along the boundary.
Under natural uniform spacelikeness and curvature assumptions, we prove a closure theorem for spacelike varifolds with boundary.
We then define spacelike Brakke flows by means of the corresponding pseudo-Euclidean Brakke inequality, and prove a compactness theorem for sequences of spacelike Brakke flows with fixed boundary. Basic monotonicity properties are established, reflecting the sign structure of the ambient indefinite metric. Finally, we adapt Ilmanen's elliptic regularization procedure \cite{Ilmanen94} to prove existence of spacelike Brakke flows, and the local regularity theorem of White \cite{White05} to the pseudo-Euclidean setting. The results provide a varifold setting for studying weak spacelike mean curvature flow beyond singularities.

\end{abstract} 

\tableofcontents

\section{Introduction}

Mean curvature flow is one of the most natural geometric evolutions for submanifolds of Riemannian manifolds. In its classical form, a family of immersions $F_t : M \to N$ evolves by the equation
\[
\frac{\partial F}{\partial t} = H,
\]
where $H$ denotes the mean curvature vector. Even when the initial submanifold is smooth, singularities may form in finite time, and one is led to weak formulations of the flow. In the Euclidean setting, Brakke's formulation in terms of time-dependent integral varifolds has proved to be a fundamental tool for studying mean curvature flow beyond singularities; see Brakke's original work \cite{Brakke78} and the subsequent developments by Ilmanen \cite{Ilmanen94}, White \cite{White05, White21}, Edelen \cite{Edelen17}, Tonegawa \cite{Tonegawa19} and many others. The corresponding compactness and closure properties are rooted in the foundational work of Allard \cite{Allard72}, Brakke \cite{Brakke78} and Ilmanen \cite{Ilmanen94}.

The purpose of this paper is to develop a weak varifold framework for spacelike mean curvature flow in pseudo-Euclidean space. More precisely, we work in
\[
\mathbf R^{n,m}=(\mathbf R^{n+m},\langle\cdot,\cdot\rangle),
\]
where the metric has signature $(n,m)$, and we consider spacelike $n$-dimensional submanifolds. In the smooth setting, spacelike mean curvature flow has been studied extensively, in particular through parabolic methods for spacelike graphs and slices; see, for instance, the work of Ecker and Huisken \cite{EH91, Ecker97}, and Li and Salavessa \cite{LiSalavessa11}. In the case of spacelike hypersurfaces in Lorentzian signature ($m=1$), Ecker \cite{Ecker03} strengthened some results using the inherent causal properties of the ambient space. These properties are harder to apply to varifolds and, in any case, they do not hold for $m>1$, underlining the power of the measure perspective. More recently, Lambert and Lotay \cite{LambertLotay21} developed a detailed analysis of spacelike mean curvature flow with boundary, including important gradient estimates and tensorial identities that are specific to the indefinite ambient metric. Our aim here is complementary: we introduce a weak formulation, in the spirit of Brakke flow, that is stable under natural compactness assumptions and that includes fixed spacelike boundary.

There are several basic difficulties that distinguish the spacelike semi-Riemannian setting from the classical Riemannian one. First, the Grassmannian of spacelike $n$-planes in $\mathbf R^{n,m}$ is noncompact. Thus, a sequence of spacelike submanifolds with locally bounded area need not have a spacelike varifold limit unless one imposes a quantitative uniform spacelikeness condition. In this paper this role is played by the gradient function $v^2$, which controls how far the tangent planes are from the light cone. A uniform bound for $v^2$ on compact subsets ensures that the relevant tangent planes remain in compact subsets of the spacelike Grassmannian.

Second, the area measure naturally induced by the pseudo-Euclidean metric differs from the Euclidean Hausdorff measure. If a spacelike rectifiable set is represented locally as a Lipschitz embedding, then its spacelike area density is given by the square root of the determinant of the induced pseudo-Euclidean metric. This density is absolutely continuous with respect to the Euclidean Hausdorff measure, but it may degenerate as the tangent planes approach the light cone. This is another reason why a uniform spacelikeness assumption is essential in the compactness theory.

Third, the sign of the evolution is different from the Riemannian case. Along a smooth spacelike mean curvature flow $(M_t)_{t \in [0,T)}$ one has
\[
\frac{d}{dt}\int_{M_t}\phi\, d\mu_t
=
\int_{M_t}
\big(
-\phi\langle H,H\rangle
+\langle \nabla\phi,H\rangle
+\partial_t\phi
\big) \,d\mu_t .
\]
Since $H$ is normal to a spacelike submanifold, the quantity $\langle H,H\rangle$ is nonpositive. Thus the term $-\langle H,H\rangle$ is nonnegative, and the basic mass monotonicity has the opposite sign from the usual Riemannian Brakke flow. This sign change is reflected in our definition of spacelike Brakke flow.

We begin by introducing spacelike integer rectifiable Radon measures and their associated spacelike varifolds, in a manner analogous to \cite{BNO12}. Given a spacelike rectifiable set $M$ with integer multiplicity, the associated varifold records both the spacelike area measure and the approximate tangent plane. The first variation is defined by integrating the semi-Riemannian tangential divergence against the measure. In the smooth case, the divergence theorem gives
\[
\int_M \operatorname{Div}_M X \,d\sigma^n
=
-\int_M \langle X,H\rangle \,d\sigma^n
+
\int_{\partial M}\langle X,\eta\rangle \,d\sigma^{n-1},
\]
where $\eta$ is the outward conormal along the boundary. This identity motivates our weak definition of generalized mean curvature and generalized boundary measure.

A point that requires care is the boundary term. In the Riemannian theory, the boundary contribution is represented by a vector-valued measure supported on the fixed boundary. In the spacelike semi-Riemannian setting, the boundary conormal should remain spacelike and should not degenerate into a null direction. We therefore introduce a class $S\mathscr{V}_n(U,\Gamma)$ of spacelike integer $n$-varifolds with boundary $\Gamma$ in an open subset $U$ of $\Rnm$, for which the first variation decomposes as
\[
\delta V(X)
=
-\int \langle H,X\rangle \, d\mu
+
\int_\Gamma \langle \nu,X\rangle \, d\sigma^{n-1},
\]
where $H$ is normal to the varifold and $\nu$ is a weak spacelike conormal along $\Gamma$ satisfying
\[
0<\langle \nu,\nu\rangle\leq 1
\]
wherever $\nu\neq 0$. This condition is stable under the convergence hypotheses used later and rules out loss of spacelikeness at the boundary.

The first main result of the paper is a closure theorem for spacelike integer varifolds with boundary: 
\begin{theorem}
	Let $(\mu_i)_{i \in \mathbf{N}} \subset S\mathscr{V}_n(U,\Gamma)$, where $\Gamma$ is a smooth spacelike $(n-1)$-dimensional submanifold of $U$. Suppose that
	the $\mu_i$ converges weakly to a Radon measure $\mu$, 
	\begin{equation*}
		d_K = \sup_i \int_K -\lnorm{H_i}\, d\mu_i < \infty, 
	\end{equation*}
	\begin{equation*}
		g_K = \sup_i \sup \{v^2(\mu_i,x) : x \in K \} < \infty, 
	\end{equation*}
	\begin{equation*}
		n_K = \inf_i\inf \{ \lnorm{\nu_i(x)} : x \in K \cap \Gamma,\, \nu_i(x) \neq 0 \} > 0. 
	\end{equation*}
	for every $K \Subset U$. Then $\mu \in S\mathscr{V}_n(U,\Gamma)$.  
\end{theorem}

This result is the spacelike analogue of the classical compactness and closure theory for integral varifolds, with additional hypotheses that prevent collapse toward the light cone.

We then define spacelike Brakke flows with fixed boundary.

\begin{definition}
	An $n$-dimensional spacelike Brakke flow with boundary in an open subset $U$ of $\Rnm$ is a pair $((\mu_t)_{t \in I},\Gamma)$ where $\Gamma$ is a smooth, properly embedded $(n-1)$-dimensional spacelike submanifold of $U$ and where $(\mu_t)_{t \in I}$ is a family of Radon measures such that:
	\begin{enumerate}[(1)]
		\item For almost every $t \in I$, $\mu_t \in S\mathscr{V}_n(U,\Gamma)$. 
		\item If $[a,b] \subset I$ and if $K \subset U$ is compact, then 
		\begin{equation*}
			\int_a^b \int_K \left( 1 - \lnorm{H} \right)\,d\mu_t dt < \infty. 
		\end{equation*}
		\item If $[a,b] \subset I$ and $\phi \in C_c^2(U \times [a,b];\mathbf{R}^{+})$, then
		\begin{align*}
			\mu_b(\phi) - \mu_a(\phi)
			&\geq \int_a^b\int \left(-\phi \lnorm{H} + \esc{\nabla \phi}{H} + \frac{\partial \phi}{\partial t}\right) \, d\mu_t dt.
		\end{align*}
	\end{enumerate}
\end{definition}
 
 This definition agrees with the smooth spacelike mean curvature flow when the latter exists.

As a basic consequence of the definition, we prove a spacelike version of a standard monotonicity property used in Brakke's compactness theory. For each $\phi \in C_c^2(U;\R^+)$ such that  
\[
K(v^2,[a,b]) := \sup_{t \in [a,b]}\int_{\spt\phi}v^2 \,d\mu_t < \infty,
\]
the function
\[
t \in [a,b] \mapsto \mu_t(\phi) + ct
\]
is monotone increasing for a suitable constant $c$ depending on $\phi$ and $K(v^2,[a,b])$. In particular, under the corresponding global assumption, the total mass is nondecreasing along the flow. This reflects the causal character of the mean curvature term and contrasts with the usual Riemannian Brakke flow.

The second main result is a compactness theorem for spacelike Brakke flows with fixed boundary. 

\begin{theorem}
	For each $i \in \mathbf{N}$, let $t \in [a,b] \mapsto \mu_t^i$ be an $n$-dimensional spacelike Brakke flow in $U$ with boundary $\Gamma$. Suppose
	\begin{equation*} 
		c_K = \sup_i \sup_{t \in [a,b]} \mu_t^i(K) < \infty,
	\end{equation*}
	\begin{equation*} 
		g_K = \sup_i \sup_{t \in [a,b]} \sup \{ v_i^2(x,t) : x \in K\} < \infty, 
	\end{equation*}
	and
	\[ 
	n_K =  \inf_i \inf_{t \in [a,b]} \inf \{\lnorm{\nu_i(x,t)} : x \in K \cap \Gamma, \, \nu_i(x,t) \neq 0 \}  > 0 
	\]
	for every $K \Subset U$. Then, after passing to a subsequence, $\mu_t^i$ converges to a Radon measure $\mu_t$ for every $t \in [a,b]$, and $t \in [a,b] \mapsto \mu_t$ is an $n$-dimensional spacelike Brakke flow in $U$ with boundary $\Gamma$
\end{theorem}

The third main result is an existence theorem for spacelike Brakke flow by means of Ilmanen's elliptic regularization procedure. 

\begin{theorem} Let $\Gamma$ be a smooth, properly embedded, $(n-1)$-dimensional, spacelike submanifold of $\Rnm$, and let $M_0$ be a spacelike $n$-rectifiable subset of $\Rnm$. Suppose that
	\begin{enumerate}[(1)]
		\item $\partial M_0 = \Gamma$.
		\item $\sup_{x \in K} v^2(M_0,x) < \infty$, for every $K \Subset \Rnm$. 
		\item $M_0 = \opn{graph}(u_{M_0})$ for some Lipschitz function $u_{M_0} : \Omega \Subset \R^n \to \R^m$.
	\end{enumerate} 
	Then there exists an $n$-dimensional spacelike Brakke flow
	\[
	t \in [0,\infty) \mapsto \mu_t
	\]
	with boundary $\Gamma$ such that 
	\[
	\mu_t \rightharpoonup \sigma^n \resmes M_0
	\]
	as $t \to 0$. 
\end{theorem}

Finally, we prove that spacelike Brakke flows with boundary satisfy a pseudo-Euclidean Huisken monotonicity formula under mild hypothesis, which allows us to adapt White's local regularity theorem to the pseudo-Euclidean setting. Some of this results can be summarized as follows:

\begin{quote}
	A spacelike Brakke flow with boundary satisfy a pseudo-Euclidean Huisken monotonicity formula (Theorem \ref{th:Monotonicity}), admits tangent flows at every spacetime point (Proposition \ref{prop:TangentFlows}) and obeys local regularity theorems in the interior and at the boundary (Theorem \ref{th:LocalRegularityII}), under the corresponding smoothness, flat-density and uniform-spacelikeness assumptions. 
\end{quote}

The organization of the paper is the following. In Section 2 we fix notation and recall the basic semi-Riemannian conventions used throughout the paper. In Section 3 we introduce spacelike rectifiable measures, spacelike varifolds, the gradient function, and the first variation of a spacelike varifold. We prove the spacelike varifold closure theorem in Section 3.4. Later, in Section 4 we define spacelike Brakke flows with boundary, establish the basic monotonicity property, and prove the compactness theorem for spacelike Brakke flows. In Section 5, we adapt Ilmanen’s elliptic regularization to prove existence of spacelike Brakke flows. Finally, in Section 8 we present a pseudo-Euclidean Huisken monotonicity formula, which will be fundamental for adapting White's local regularity theorem to the pseudo-Euclidean case in the last Section. 

\medskip
\noindent \textbf{Acknowledgments}. I would like to express my gratitude to my advisors Francisco Martín and Miguel Sánchez. I am also profoundly thankful to Brian White for their thoughtful suggestions and corrections.

\section{Preliminaries}

Let $\R^{n,m} = (\R^{n+m}, \lnorm{\cdot})$ be the pseudo-Euclidean space, where
\[
\esc{u_1}{u_2} = \sum_{i=1}^n u_1^i u_2^i - \sum_{i=n+1}^{n+m} u_1^iu_2^i.
\]
for any $u_1,u_2 \in \Rnm$. We recall that a vector $u \in \Rnm$ is called:
\begin{itemize}
	\item spacelike if $\lnorm{u} > 0$ or $u = 0$,
	\item null if $\lnorm{u} = 0$ and  $u \neq 0$,
	\item timelike if $\lnorm{u} < 0$.
\end{itemize}
If $u \in \Rnm$ is either spacelike or null, we let
\[
\abs{u} = \sqrt{\lnorm{u}} \geq 0.
\]

We want to generalize the theory of varifolds in the semi-Riemannian setting, in order to take limits of spacelike submanifolds under weak curvature bounds. 

For an open subset $U$ of $\Rnm$, we need first to consider the product space defined by 
\[
G_n(U) \coloneqq U \times G(n,m),
\]
where $G(n,m)$ denotes the set of $n$-dimensional spacelike planes of $\Rnm$, which is noncompact. We let $\mathfrak{X}(U)$ denote the space of continuous, compactly supported vectorfields on $U$, and $\mathfrak{X}_n(U)$ denote the space of continuous, compactly supported functions on $G_n(U)$ that assign to each $(x,S)$ in $G_n(U)$ a vector in $\Tan(U,x)$.

We say that $V$ is a spacelike $n$-varifold if $V$ is a positive Radon measure on $G_n(U)$. We typically do not consider this completely general and abstract spacelike $n$-varifold, though it is important when we consider a convergence in this topology. The most important example of a spacelike $n$-varifold is the one naturally induced from a smooth, properly embedded, $n$-dimensional, spacelike submanifold $M$ of $U$. Due to Riesz representation theorem \cite[Theorem 4.1]{Simon83}, a Radon measure is uniquely determined once we give a locally bounded linear map $L \colon C_c(G_n(U)) \rightarrow \R$. Given $\phi \in C_c(G_n(U))$, whose independent variables are $x \in U$ and $S \in G(n,m)$, the most natural thing to compute is
\[
\int_M \phi(x,\Tan(M,x))\, d\mu_M(x),
\]
where $\mu_M$ is the volume measure on $M$. This assignment for $\phi \in C_c(G_n(U))$ is linear and, for each pair of compact sets $K \subset U$ and $\mathcal{T} \subset G(n,m)$, and for all $\phi$ with $\spt \phi \subset K \times \mathcal{T}$,
\[
\int_M \phi(x,\Tan(M,x)) \, d\mu_M(x) \leq \operatorname{area}(M \cap K)\sup_{x \in K,\, S \in \mathcal{T}}\abs{\phi(x,S)}
\]
so it defines a locally bounded linear functional. Therefore, we define the Radon measure $\abs{M}$ on $G_n(U)$ by the rule
\[
\int_{G_n(U)} \phi(x,S) \, d\abs{M}(x,S) \coloneqq \int_M \phi(x,\Tan(M,x)) \, d\mu_M(x)
\]
for all $\phi \in C_c(G_n(U))$.

There are many reasons to consider this abstract setting. Suppose that we have a sequence of spacelike $n$-dimensional submanifolds $M_i$ which have locally uniformly finite areas, that is, $\sup_i \opn{area}(M_i \cap K) < \infty$ for any $K \Subset U$. Each $M_i$ determines a Radon measure, defined by $\mu_i(A) = \opn{area}(M_i \cap A)$ for any Borel subset A of $U$, and the local area bound implies that $(\mu_i)$ is locally uniformly bounded. Then, by the compactness theorem of Radon measures (see \cite[Theorem 4.4]{Simon83}), there exists a subsequence $\mu_{i(j)}$ and a limit Radon measure $\mu$ on $U$ such that 
\[
\int_U \phi(x)\,d\mu_{i(j)}(x) \rightarrow \int_U \phi(x)\,d\mu(x)
\]
for all $\phi \in C_c(U)$. Moreover, suppose that the spacelike $n$-dimensional submanifolds are locally uniformly spacelike, that is, for any compact subset $K \subset U$, there exists a compact subset $\mathcal{T} \subset G(n,m)$ such that $\Tan(M_i,x) \in \mathcal{T}$ for all $i$ and $x \in K$. If we consider the sequence of spacelike $n$-varifolds $\abs{M_i}$ instead, the sequence also has locally uniformly finite measure on $G_n(U)$. By the compactness theorem of Radon measures again, there exists a
converging subsequence and a limit spacelike $n$-varifold $V$ such that
\[
\int_{G_n(U)} \phi(x,S)\,d\abs{M_{i(j)}}(x,S) \rightarrow \int_{G_n(U)} \phi(x,S)\,dV(x,S)
\]
for all $\phi \in C_c(G_n(U))$. From this, we see that $V$ contains more information than $\mu$ regarding the behavior of tangent planes of $\abs{M_i}$. We will see more useful properties in Section \ref{Section:SpacelikeVarifolds}. 

\section{Spacelike varifolds}\label{Section:SpacelikeVarifolds}

\subsection{Geometric measure theory}

We denote by $\mathcal{H}^n$ the euclidean $n$-dimensional Hausdorff measure in $\R^{n+m}$. Let $U$ be an open subset of $\R^{n+m}$ and let $M$ be a $\mathcal{H}^n$-measurable subset of $U$. We say that $M$ is countable $n$-rectifiable (``$n$-rectifiable'' for short) if $\mathcal{H}^n$-almost all of the set $M$ can be covered by a countable union of $n$-dimensional $C^1$ submanifolds of $U$. Therefore, an $n$-rectifiable set admits tangent space $\mathcal{H}^n$-almost everywhere. 

We let $\Tan(M,x)$ be the tangent space to $M$ at $x$ (where it is defined). The set $M$ is called spacelike if $\Tan(M,x)$ is spacelike for all $x \in M$ where $\Tan(M,x)$ exists (in particular, $\mathcal{H}^n$-almost everywhere on $M$). The set $M$ is called non-timelike if $\Tan(M,x)$ is either spacelike or null. 

Suppose that $M$ is a spacelike $n$-rectifiable subset of $U$, and let $\phi \in C_c^1(U)$ and $X \in C_c^1(U;\Rnm)$. We denote by $\nabla^M \phi$ the semi-Riemannian tangential gradient of $\phi$ on $M$, defined as 
\[
\nabla^M \phi \coloneq \Pi_{\Tan(M,\cdot)}\nabla\phi,
\]
where $\nabla\phi$ is the ambient gradient of $\phi$ and $\Pi_{\Tan(M,\cdot)}$ denotes the semi-Riemannian orthogonal projection onto $\Tan(M, \cdot)$. We define the semi-Riemannian tangential divergence $\Div_M X$ on $M$ as
\[ 
\Div_M X \coloneqq \sum_{i=1}^n \esc{\tau_i}{\nabla_{\tau_i}X}
\]
where $\{\tau_1,\dots,\tau_n\}$ is an orthonormal base of $\Tan(M,\cdot)$. The tangential divergence verifies
\[ 
\Div_M(\phi X) = \phi \Div_M X + \esc{\nabla^M\phi}{X}.
\]

The $n$-dimensional area $\sigma^n$ of a spacelike $n$-rectifiable subset $M = X(\Omega)$ of $\Rnm$, where $\Omega \subset \R^n$ is an open subset and $X : \Omega \to \Rnm$ is a Lipschitz embedding, is given by 
\[
\sigma^n(M) = \int_\Omega \sqrt{\det g_S|\Tan(M,\cdot)},
\]
where $g_S|\Tan(M,\cdot)$ is the  canonical pseudo-Euclidean metric of $\Rnm$ restricted to $\Tan(M,\cdot)$. We are interested in the relation between the $n$-dimensional measure $\sigma^n$ associated with a spacelike $n$-rectifiable subset $M$ of $\Rnm$ and its Hausdorff $n$-dimensional measure $\H^n$ on $\R^{n+m}$. Let $g_E$ denote the canonical euclidean metric of $\R^{n+m}$. Using the Euclidean area formula \cite[Page 125]{EvansGariepy92}, we have 
\[
\sigma^n(M) = \int_{M \cap A} \frac{\sqrt{\det g_S|\Tan(M,\cdot)}}{\sqrt{\det g_E|\Tan(M,\cdot)}}\, d\H^n.
\]
Let $x \in M$ such that $\Tan(M,x)$ exists and let $\{\tau_1,\dots,\tau_n\}$ be a $g_S$-othonormal base of $\Tan(M,x)$. Therefore, 
\[ 
\sqrt{\det g_S|\Tan(M,x)} = 1. 
\]
Denote by $\mcal{B}_{\Rnm} = \{e_1,\dots,e_n,e_{n+1},\dots,e_{n+m}\}$ the canonical base of $\Rnm$, where $e_i$ is spacelike for $i=1,\dots,n$ and timelike for $i=n+1,\dots, n+m$. Every $0 \neq u \in \Tan(M,x)$ can be written as 
\[ 
u = \sum_{i=1}^n a_ie_i + \sum_{i=n+1}^{n+m} b_ie_i 
= u_s + u_t. 
\] 
Using the definition of both metrics, we compute the Rayleigh quotient,
\begin{align*}
	\frac{g_E(u,u)}{g_S(u,u)} = \frac{\abs{u_s}^2_{\R^{n+m}} + \abs{u_t}^2_{\R^{n+m}}}{\abs{u_s}^2_{\R^{n+m}} - \abs{u_t}^2_{\R^{n+m}}} = 1 + \frac{2\abs{u_t}^2_{\R^{n+m}}}{\abs{u_s}^2_{\R^{n+m}} - \abs{u_t}^2_{\R^{n+m}}} \geq 1
\end{align*}
Therefore, the eigenvalues of the matrix $g_E(\tau_i,\tau_j)$ verify $\lambda_i \geq 1$ for every $i=1,\dots,n$ and, consequently,
\[
S(\Tan(M,x)) := \frac{\sqrt{\det g_S|\Tan(M,x)}}{\sqrt{\det g_E|\Tan(M,x)}} = \frac{1}{\sqrt{\lambda_1 \cdots \lambda_n}} \in (0,1]. 
\] 

Since $\Tan(M,x)$ exists for $\H^n$ almost every $x \in \Rnm$, we conclude 
\[ 
d\sigma^n = S(\Tan(M,\cdot))\, d\H^n,
\quad 0 < S(\Tan(M,\cdot)) \leq 1.
\]
Therefore, the $n$-dimensional spacelike area measure $\sigma^n$ is absolutely continuous with respect to $\H^n$, that is, $\sigma^n(A) = 0$ for every Borel subset $A$ such that $\H^n(A)=0$. In fact, we have that $\sigma^n \leq \H^n$. 

Following the smooth case \cite[Section 2.2]{LambertLotay21}, we introduce the gradient function. If $P \in G(n,m)$ is a spacelike $n$-dimensional plane, we define the gradient function as
\[ 
v^2(P) = \sum_{i=n+1}^{n+m}-\lnorm{e_i^\perp} = m + \sum_{i=n+1}^{n+m}\lnorm{e_i^\top} \geq m,
\]
where $e_i^\perp$ and $e_i^\top$ are the semi-Riemannian orthogonal projections of $e_i$ onto $P^\perp$ and $P$ respectively. If $M$ is a spacelike $n$-rectifiable subset of $\Rnm$, let
\[ 
v^2(M,x) := v^2(\Tan(M,x)).  
\]  
where $\Tan(M,x)$ exists, and let $v^2(M,x) \coloneqq m$ where $\Tan(M,x)$ does not exist. 

A uniform bound for the gradient function $v^2(P)$ implies a positive lower bound for $S(P)$ on compact subsets of the spacelike Grassmanian. Let $P$ be an $n$-dimensional spacelike plane and $\{\tau_1,\dots,\tau_n\}$ an orthonormal base of $P$. Again, every $u \in P$, with $\lnorm{u}=1$, can be written as 
\[ 
u = \sum_{i=1}^n \esc{u}{e_i}e_i - \sum_{i=n+1}^{n+m} \esc{u}{e_i}e_i = u_s + u_t.
\] 
Since
\[ \esc{u}{e_i}^2 \leq \lnorm{e^\top_i} = -1 + \lnorm{e^\perp_i}, 
\]
Adding for $i \in \{n+1,\dots,n+m\}$, we obtain  
\[ \abs{u_t}^2_{\R^{n+m}} \leq -m + v^2(P) 
\]
and
\[ \frac{\abs{u_t}^2_{\R^{n+m}}}{\abs{u_s}^2_{\R^{n+m}}} = \frac{\abs{u_t}^2_{\R^{n+m}}}{1 + \abs{u_t}^2_{\R^{n+m}}} \leq 1 - \frac{1}{1-m+v^2(P)} < 1.
\]
Using the definition of both metrics, with $g_S(\cdot,\cdot) = \lnorm{\cdot}$, and taking into account that the function 
\[ 
x \in [0,1) \mapsto \frac{1 + x}{1 - x}
\] 
is increasing,  
we can estimate the Rayleigh quotient by 
\[
\frac{g_E(u,u)}{g_S(u,u)} \leq 1 - 2m + 2v^2(P). 
\]
Therefore, the eigenvalues of the matrix $g_E(\tau_i,\tau_j)$ verify $\lambda_i \leq 1 - 2m + 2v^2(P)$ for every $i=1,\dots,n$ and, consequently,
\[
S(P) := \frac{\sqrt{\det g_S|P}}{\sqrt{\det g_E|P}} = \frac{1}{\sqrt{\lambda_1 \cdots \lambda_n}} \geq \frac{1}{(1 - 2m + 2v^2(P))^{n/2}}.
\]

The gradient also plays a vital role in the pointwise estimates of ambient tensors. Suppose $T$ is a $(r+s)$-covariant tensor on $\Rnm$. Defining 
\[ 
T_{i_1,\dots,i_r,j_1,\dots,j_s} = T(\tau_{i_1},\dots,\tau_{i_r},\xi_{j_1},\dots,\xi_{j_s}), 
\]
where $\{\tau_{i_1},\dots,\tau_{i_r}\}$ and $\{\xi_{j_1},\dots,\xi_{j_s}\}$ are orthonormal subsets of $\Tan(M,\cdot)$ and $\Tan(M,\cdot)^\perp$ respectively, we have that 
\begin{equation}\label{eq:GradientEstimateI}
	\abs{T_{i_1,\dots,i_r,j_1,\dots,j_s}} \leq v^{r+s}\abs{T}_{\R^{n+m}}. 
\end{equation}
by \cite[Equation 2.3]{LambertLotay21}.

\subsection{Spacelike integer varifolds}

The generalized manifolds we consider in this paper are the \textit{Spacelike integer varifolds}. In this section we will focus on defining a weak notion of mean curvature and boundary. 

\begin{definition}[Spacelike integer Radon measures]
    We say that a Radon measure $\mu=\mu(M,\theta)$ on $U$ is spacelike integer $n$-rectifiable, and we write $\mu \in SI\mathcal{M}_n(U)$, if there exists a spacelike $n$-rectifiable subset $M$ of $U$ and a multiplicity function $\theta \in L^1_{\operatorname{loc}}(M,\sigma^n)$ that takes values in the nonnegative integers such that 
	\[
	\int f(x)\, d\mu(x)
	\coloneq \int_{M} \theta(x)f(x)\, d\sigma^n(x)
	\]
    for all $f \in C_c(U)$.  
\end{definition}

Suppose $\mu=\mu(M,\theta)$ is a spacelike integer $n$-rectifiable Radon measure on $U$. We let $\opn{Var}(\mu)$ be the associated \textit{spacelike integer $n$-rectifiable varifold} in $U$, that is, the Radon measure on $G_n(U)$ defined by 
\[
\int f(x,S)\, d\opn{Var}(\mu)(x,S)
\coloneq \int_M \theta(x)f(x,\Tan(M,x))\, d\sigma^n(x)
\]
for all $f \in C_c(G_n(U))$.

\subsection{First variation of a varifold}

In the semi-Riemannian setting we have the following divergence theorem for smooth, properly embedded, spacelike submanifolds with smooth boundary, which motivates the definition of the first variation of a varifold.

\begin{theorem}[Divergence Theorem]
    Let $M$ be an $n$-dimensional, spacelike, properly embedded, oriented submanifold of class $C^2$ of $U$ with spacelike boundary $\Gamma$. If $X \in C^1_c(U;\Rnm)$, then   \begin{equation}\label{Eq:DivergenceTheorem}
    \int_M \Div_M(X)\, d\sigma^n = -\int_M \esc{X}{H}\, d\sigma^n + \int_\Gamma \esc{X}{\eta}\, d\sigma^{n-1},
\end{equation}
where $H$ denotes the mean curvature vector of $M$ and $\eta$ denotes the outer unit normal field to $M$ along $\Gamma$, which is tangent to $M$ at all boundary points.
\end{theorem}

\begin{proof}
    We can decompose the vector field $X \in C_c^1(U;\Rnm)$ as
    \[
    X = X^\top + X^\perp = X^\top - \sum_{j = 1}^m \esc{X}{\xi_j}\xi_j
    \]
    and
    \begin{align*}
    \Div_M X& =\Div_M X^\top -\sum_{j = 1}^m \Div_M(\esc{X}{\xi_j}\xi_j)\\
    &= \Div_M X^\top-\sum_{j = 1}^m \esc{X}{\xi_j} \Div_M(\xi_j)\\
    &= \Div_M X^\top-\sum_{j = 1}^m \esc{X}{\xi_j} \sum_{i=1}^n \esc{\tau_i}{\nabla_{\tau_i}\xi_j}\\
    &= \Div_M X^\top - \esc{X}{\sum_{i=1}^n II(\tau_i,\tau_i)}\\
    &= \Div_M X^\top - \esc{X}{H}.
    \end{align*}
    By the Stokes theorem on $M$,
    \begin{align*}
    \int_M \Div_MX\, d\sigma^n &= \int_M \Div_M X^\top\, d\sigma^n - \int_M \esc{X}{H}\, d\sigma^n\\
    &= \int_\Gamma \esc{X}{\eta}\, d\sigma^{n-1} - \int_M \esc{X}{H}\, d\sigma^n. 
    \end{align*} 
\end{proof}

Let $V$ be a Radon measure on $G_n(U)$, and $X \in C_c^1(U;\Rnm)$. Then, the function 
\[
(x,S) \in G_n(U) \mapsto \Div_S X,
\]
where
\[
\Div_S X = \sum_{i=1}^n \esc{\nabla_{\tau_i} X}{\tau_i}
\]
for $\{\tau_1,\dots,\tau_n\}$ an orthonormal base of $S$, is in $C_c(G_n(U))$. In particular, if $V = \Var(\mu)$ for some spacelike integer $n$-dimensional Radon measure $\mu = \mu(M,\theta)$, we write 
\[
\Div_\mu X = \sum_{i=1}^n \esc{\nabla_{\tau_i} X}{\tau_i}
\]
for $\{\tau_1,\dots,\tau_n\}$ an orthonormal base of $\Tan(M,x)$. In a similar way to the Riemannian case, we can give the following definition. 

\begin{definition}[First variation of a varifold]
	If $\mu$ is a spacelike integer $n$-rectifiable Radon measure on $U$, we define the first variation of $V=\opn{Var}(\mu)$ by 
	\[ 
	\delta V(X) := \int_{U} \Div_{\mu} X \,d\mu
	\]
	for all $X \in C_c^1(U;\Rnm)$. 
\end{definition} 

The following definition is also the same as in the Riemannian case. 

\begin{definition}
	If $\mu$ is a spacelike integer $n$-rectifiable Radon measure on $U$, we say that $V=\opn{Var}(\mu)$ is \textit{stationary} if 
    \[
    \delta V(X) = 0
    \]
    for all $X \in C_c^1(U;\Rnm)$. More generally, we say $V$ has \textit{bounded first variation} if there exists $c_K>0$ such that 
    \[
    \abs{\delta V(X)} \leq c_K \sup_K \abs{X}_{\R^{n+m}}
    \]
    for any $X \in C_c^1(U;\R^{k,n})$ with $\spt X \subset K \Subset U$.  
\end{definition}

\begin{remark} 
	Suppose $(\mu_i)_{i \in \mathbf{N}}$ and $\mu$ are spacelike integer $k$-rectifiable Radon measures in $U$ such that $\opn{Var}(\mu_i)$ converges to $\Var(\mu)$ as Radon measures. Then, 
    \[
    \delta \opn{Var}(\mu_i)(X) \rightarrow \delta \opn{Var}(\mu)(X)
    \]
    for all $X \in C_c^1(U;\Rnm)$. 
\end{remark} 

We would to define a weak notion of  mean curvature vector for a spacelike integer varifold with locally bounded first variation, that will be consistent with our semi-Riemannian ambient. In order to do this, we need to use the Hilbert space structure of $(\R^{n+m}, \lnorm{\cdot}_{\R^{n+m}})$, where $\lnorm{\cdot}_{\R^{n+m}}$ is the canonical metric of $\R^{n+m}$ and $\enorm{\cdot}$ its associated norm. 

Let $\mu$ be a spacelike integer $n$-rectifiable Radon measure in an open subset $U$ of $\Rnm$ with locally bounded first variation, and let $V=\opn{Var}(\mu)$. Approximating by $C_c^1(U;\Rnm)$ vector fields, we can extend the definition of $\delta V$ to $C_c(U;\Rnm)$. Since $\delta V$ is a locally bounded linear functional on $C_c(U;\Rnm)$, by Riesz representation theorem \cite[Theorem 4.1]{Simon83}, there exists a Radon measure $\abs{\delta V}$ and a $\abs{\delta V}$-measurable function $\tilde{\eta} : U \rightarrow \R^{n+m}$ such that $\abs{\tilde{\eta}(x)}_{\R^{n+m}} = 1$ for $\abs{\delta V}$-almost every $x \in U$ and 
\[
\delta V(X) = \int \esc{X}{\tilde{\eta}}_{\R^{n+m}} \, d\abs{\delta V}
\]
for all $X \in C_c(U;\Rnm)$. Moreover, $\tilde{\eta}$ is unique up to sets of $\abs{\delta V}$-measure zero. Now, we apply the Radon–Nikodym theorem \cite[Theorem 1.28]{AmbrosioFuscoPallara00} to the pair $\abs{\delta V}$ and $\mu$. The absolutely continuous part of $\mu$ is represented as a locally integrable and $\mu$-measurable function $f$. Writing the singular part as $\beta = \beta(\mu)$, we have
\[
d\abs{\delta V} = f d\mu + d\beta.
\]
Combining the Riesz representation theorem with Radon-Nikodym decomposition, we obtain
\[
	\delta V(X) 
	= \int f \esc{X}{\tilde{\eta}}_{\R^{n+m}} \, d\mu + \int \esc{X}{\tilde{\eta}}_{\R^{n+m}} \, d\beta.\\
\]
Let 
\[
    \xi = \opn{diag}(1,\dots,1-1,\dots,-1)
\]
be the semi-Riemannian metric tensor of $\Rnm$, and define
$H = -\xi f \tilde{\eta}$ and $\eta = \xi\tilde{\eta}$, then 
\[
	\delta V(X) = -\int \esc{X}{H} \, d\mu + \int \esc{X}{\eta} \, d\beta.
\]

We now consider a simple example to better understand the abstract concepts in the previous construction. 

\begin{example}
	In Lorentz-Minkowski space $\L^2 = (\R^2, dx^2-dy^2)$, let 
	\[
		\gamma(t) = (x(t),y(t)), \quad t \in [a,b],
	\]
	be a spacelike curve such that $\lnorm{\gamma'(t)} = 1$. Consider the spacelike integer $1$-varifold given by $\gamma$, that is, 
	\[
		\int_{G_1(\mathbf{L}^2)} \phi(x,S)\, d\opn{Var}(\gamma)(x,S) = \int_a^b \phi(\gamma(t),\Tan(M,\gamma(t)))\sqrt{\lnorm{T(t)}}\, dt,
	\] 
	where $T(t) = \gamma'(t)$. If $X = (u,v) \in C_c^1(\mathbf{L}^2;\mathbf{L}^2)$, then $\nabla_{T}X = (T(u),T(v))$ and 
	\begin{align*}
		\delta \opn{Var}(M) &= \int \Div_\gamma X \, d\gamma\\
		&= \int_\gamma \esc{T}{\nabla_T X}\, d\sigma^1\\
		&= \int_a^b \left[ x'(t)T(u)-y'(t)T(v) \right] \, dt 
	\end{align*}
	Since $T(u) = \partial_t u(x(t),y(t))$ and $T(v) = \partial_t v(x(t),y(t))$, integrating by parts, we obtain
	\begin{align*}
		\delta \opn{Var}(M) &= \int_a^b \left[ x'(t)\partial_t u(t) - y'(t)\partial_t v(t) \right] \, dt\\
		&= \left[x'(t)u(t) \right]_{t=a}^{t=b} - \int_a^b x''(t)u(t)\, dt - \left[y'(t)v(t) \right]_{t=a}^{t=b} + \int_a^b y''(t)v(t)\, dt\\ 
		&= \int_a^b \left[ -x''(t)u(t) + y''(t)v(t) \right] \, dt + (x'u-y'v)
        (b) - (x'u - y'v)(a)\\
        &= \int_a^b \left[ -x''(t)u(t) + y''(t)v(t) \right] \, dt + \int (x'u-y'v)(t)\, d\delta_b - \int (x'u - y'v)(t) \, d\delta_a
	\end{align*}
	where $\delta_a$ and $\delta_b$ are Dirac measures. Finally, let
    \[
        \mathcal{T} = \{t \in (a,b) : \abs{(-x''(t),y''(t)}_{\R^2} \neq 0 \}
    \]
    and 
	\[ 
	\tilde{\eta}(t) = 
	\begin{cases}
		\frac{(-x''(t), y''(t))}{\abs{(-x''(t), y''(t))}_{\R^2}}, &  t \in \mathcal{T} ,\\
		(x'(t),-y'(t)), &  t=b,\\
		(-x'(t),y'(t)), &  t=a, \\
		(0,0), & \text{otherwise}.
	\end{cases} 
	\]
	Using the previous notation, we have that $f = \abs{(-x'',y'')}_{\R^2}$ and $\beta = \delta_a + \delta_b$. Therefore, 
	\[ 
	H = (x'',y''), \quad \eta(a) = -T(a), \quad \text{and} \quad \eta(b) = T(b). 
	\]
	
\end{example}

\begin{definition}
	Let $\Gamma$ be a properly embedded $(n-1)$-dimensional spacelike submanifold of $U$. We denote by $S\mathscr{V}_n(U,\Gamma)$ the set of spacelike integer $n$-rectifiable Radon measures $\mu$ in $U$ such that $\mu$ has locally bounded first variation, that is,  
	\begin{equation}\label{eq:DivergenceTheorem}
		\int \Div_\mu X\, d\mu = -\int \esc{H}{X}\, d\mu + \int \esc{X}{\eta}\, d\beta. 
	\end{equation}
	for all $X \in C_c^1(U;\Rnm)$, and such that
	\begin{enumerate}[(1)]
		\item $H(x) \in \Tan(\mu,x)^\perp$ for $\mu$-almost every $x \in U$.
		\item $\eta(x) \in \Tan(\Gamma,x)^\perp$ is spacelike for $\beta$-almost every $x \in U$.
		\item $\sqrt{\lnorm{\eta}}\beta \leq \sigma^{n-1} \resmes \Gamma$ and, in particular, $\spt\beta \subset \Gamma$.   
	\end{enumerate}
\end{definition}

Therefore, if $\mu \in S\mathscr{V}_n(U,\Gamma)$, by analogy with the classical identity \eqref{Eq:DivergenceTheorem}, we call 
\[ 
H, \quad \hat{\beta} = \sqrt{\lnorm{\eta}} \beta, \quad \text{and} \quad \hat{\eta} = \eta / \sqrt{\lnorm{\eta}}  
\]
the \textit{generalized mean curvature}, the \textit{generalized boundary measure} and the \textit{generalized unit co-normal} of $\mu$, respectively.

\begin{remark}
	If $M$ is a smooth, properly embedded, $n$-dimensional, spacelike submanifold of $U$, the boundary vector $\hat{\eta}$ is the usual outward conormal, that is,
	it is tangent to $M$ and orthogonal to $\Gamma=\partial M$. In the weak
	definition above we only encode the boundary contribution as a spacelike
	vector-valued measure supported on $\Gamma$. Thus $\eta$ should be understood
	as a generalized boundary vector. When an appropriate boundary tangent trace
	exists, one may additionally require that
	\[
	\Tan(\Gamma,x) \oplus \R\eta(x)
	\]
	coincides with the boundary trace of the tangent planes of $\mu \in S\mathscr{V}_n(U,\Gamma)$.
\end{remark}

The following example illustrates that the term $\sqrt{\lnorm{\eta}}$ acts as an estimate on how spacelike is $\mu$ at $\Gamma$. 

\begin{example}
    In the Lorentz-Minkowski space $\L^3 = (\R^3, dx^2 + dy^2 -dz^2)$, let $M$ be the smooth spacelike  hypersurface with boundary $\Gamma$ constructed by rotating the curve 
    \[
    z = y^2, \quad y \in [-1/2,1/2],
    \]
    around the $z$-axis. Then, we have that 
    \[
    \Gamma = \{(x,y,z) \in \L^3 : x^2 + y^2 = 1/4,\, z=1/4\}
    \]
    and $d\Gamma = d(\H^1 \resmes \Gamma)$, but $ \sqrt{\lnorm{\eta}}\beta = 0$, since the tangent planes of $M$ are asymptotic to the light cone at $\Gamma$.   
\end{example}

Let 
\[ 
	\nu(\mu,x) = \lim_{r \to 0} \frac{\int_{B_r(x)}\eta\,d\beta}{(\sigma^{n-1} \resmes \Gamma)(B_r(x))} 
\]
where the limit exists, and let $\nu(\mu,x) = 0$ where the limit does not exist. By Radon-Nikodym Theorem \cite[Theorem 4.7]{Simon83}, the limit exists $(\sigma^{n-1} \resmes \Gamma)$-almost everywhere and we can rewrite \eqref{eq:DivergenceTheorem} as 
\[ 
	\int \Div_\mu X\, d\mu = -\int \esc{H}{X}\, d\mu + \int_\Gamma \esc{\nu}{X}\, d\sigma^{n-1}. 
\]
or, using $\Gamma$ to denote its associated Radon measure, as 
\[ 
	\int \Div_\mu X\, d\mu = -\int \esc{H}{X}\, d\mu + \int \esc{\nu}{X}\, d\Gamma. 
\]

\begin{lemma}\label{lemma:ConormalEquivalence}
	The condition that $\sqrt{\lnorm{\eta}}\beta \leq \sigma^{n-1} \resmes \Gamma$ is equivalent to the condition that
	\[
		0 < \lnorm{\nu} \leq 1,
	\]
	$\sigma^{n-1}$ almost everywhere on $\{x \in \Gamma : \nu(x) \neq 0\}$ and
	\[
		\int \Div_\mu X\, d\mu = -\int \esc{H}{X}\, d\mu + \int_\Gamma \esc{\nu}{X}\, d\sigma^{n-1},
	\]
	for every $X \in C_c^1(U;\Rnm)$.  
\end{lemma}

\begin{proof}
	If $\sqrt{\lnorm{\eta}}\beta \leq \sigma^{n-1} \resmes \Gamma$, by Radon-Nikodym Theorem, there exists a $\sigma^{n-1} \resmes \Gamma$ measurable function $0 \leq f \leq 1$ such that $\sqrt{\lnorm{\eta}}\beta = f(\sigma^{n-1} \resmes \Gamma)$ and, consequently, 
    \[
        \eta\beta = \frac{\eta f}{\sqrt{\lnorm{\eta}}}(\sigma^{n-1} \resmes \Gamma).
    \]
    By the definition of the Radon-Nikodym derivative, 
    \[
        \nu(x) = \frac{\eta(x)f(x)}{\sqrt{\lnorm{\eta(x)}}}
    \]
    for $\sigma^{n-1}$ almost every $x \in \Gamma$. Therefore, $0 < \lnorm{\nu} \leq  1$ for $\sigma^{n-1}$ almost everywhere on $\{x \in \Gamma : \nu(x) \neq 0\}$ and 
	\[ 
	\int \Div_\mu X\, d\mu = -\int \esc{H}{X}\, d\mu + \int_\Gamma \esc{\nu}{X}\, d\sigma^{n-1} 
	\] 
	for every $X \in C_c^1(U;\Rnm)$.
	
	On the other hand, if $0 < \lnorm{\nu} \leq  1$ for $\sigma^{n-1}$ almost everywhere on $\{x \in \Gamma : \nu(x) \neq 0\}$ and
	\[ 
	\int \Div_\mu X\, d\mu 
	= -\int \esc{H}{X}\, d\mu + \int \esc{\eta}{X}\, d\beta 
	= -\int \esc{H}{X}\, d\mu + \int_\Gamma \esc{\nu}{X}\, d\sigma^{n-1}, 
	\]
	then $\eta\beta = \nu \sigma^{n-1} \resmes \Gamma$ as vector measures. Therefore, there exists a $\beta$-measurable function $b \geq 0$ such that $\nu = b\eta$ and, consequently,
	\[ 
		\sqrt{\lnorm{\eta}}\beta 
		= b \sqrt{\lnorm{\eta}}\sigma^{n-1} \resmes \Gamma
		= \sqrt{\lnorm{\nu}} \sigma^{n-1} \resmes \Gamma \leq \sigma^{n-1} \resmes \Gamma. 
	\]
\end{proof}

\subsection{Compactness of Integer Spacelike Varifolds}

In the following theorem, we use the notation $H_i(\cdot) = H(\mu_i,\cdot)$, $H(\cdot) = H(\mu,\cdot)$, $\hat{\eta}_i(\cdot) = \eta(\mu_i,\cdot)/\sqrt{\lnorm{\eta(\mu_i,\cdot)}}$ and $\hat{\beta}_i = \sqrt{\lnorm{\eta(\mu_i,\cdot)}}\beta(\mu_i)$, $\nu_i(\cdot) = \nu(\mu_i,\cdot)$ and $\nu(\cdot) = \nu(\mu,\cdot)$.

\begin{theorem}[Spacelike Varifold Closure Theorem]\label{CTRV}
	Let $(\mu_i)_{i \in \mathbf{N}} \subset S\mathscr{V}_n(U,\Gamma_i)$, where $\Gamma_i$ are smooth spacelike $(n-1)$-dimensional submanifold of $U$ that converge in $C^\infty$ to a smooth spacelike submanifold $\Gamma$. Suppose that the $\mu_i$ converges weakly to a Radon measure $\mu$,
	\begin{equation}\label{eq:UniformMeanCurvature}
		d_K = \sup_i \int_K -\lnorm{H_i}\, d\mu_i < \infty, 
	\end{equation}
	\begin{equation}\label{eq:UniformSpacelike}
		g_K = \sup_i \sup \{v^2(\mu_i,x) : x \in K \} < \infty, 
	\end{equation}
	\begin{equation}\label{eq:UniformConormal}
		n_K = \inf_i\inf \{ \lnorm{\nu_i(x)} : x \in K \cap \Gamma_i,\, \nu_i(x) \neq 0 \} > 0. 
	\end{equation}
	for every $K \Subset U$. Then
	\begin{enumerate}[(1)]
		\item For every $K \Subset U$, 
		\[ 
			\sup_i \parent{\int_K \sqrt{-\lnorm{H_i}}\, d\mu_i + \hat{\beta}_i(K)} < \infty. 
		\]
		\item $\mu \in SI\mathcal{M}_n(U)$ and $\Var(\mu_i)$ converges to $\opn{Var}(\mu)$. Thus, if $f \in C_c(G_n(U))$, then  
		\[ \int f(x,\Tan(\mu_i,x))\, d\mu_i(x) \rightarrow \int f(x,\Tan(\mu,x))\, d\mu(x) \]
		\item If $X \in \mathfrak{X}_n(U)$, then 
		\[ 
			\int \esc{X(x,\Tan(\mu_i,x))}{H_i(x)}\, d\mu_i(x) \rightarrow \int \esc{X(x,\Tan(\mu,x))}{H(x)}\, d\mu(x) 
		\]
		\item If $Z \in \mathfrak{X}_{n-1}(U)$, then 
		\[ 
			\int \esc{Z(x,\Tan(\Gamma_i,x))}{\nu_i(x)}\, d\Gamma_i(x) \rightarrow \int \esc{Z(x,\Tan(\Gamma,x))}{\nu(x)}\, d\Gamma(x) 
		\]
		\item $\mu \in S\mathscr{V}_n(U,\Gamma)$.  
	\end{enumerate}  
\end{theorem}

\begin{proof}
	Since the $\mu_i$ converge to $\mu$, 
	\[
		c_K \coloneq \sup_i \mu_i(K) < \infty
	\]
	for every $K \Subset U$. Thus
	\[
		\int_K \sqrt{-\lnorm{H_i}}\, d\mu_i \leq (\mu_i(K))^{1/2}\left( \int_K -\lnorm{H_i}\, d\mu_i \right)^{1/2} \leq (c_K d_K)^{1/2}
	\]
	Also,
	\[ 
		\sup_i \hat{\beta}_i(K) \leq \sup_i \sigma^{n-1}(\Gamma_i \cap K) < \infty. 
	\]
	since $\Gamma_i$ converges to $\Gamma$ in $C^\infty$. Therefore, by \eqref{eq:GradientEstimateI}, 
	\begin{align}\label{eq:AllardHypothesis}
		\abs{\delta \opn{Var}(\mu_i)(X)} &\leq \int_K \abs{\esc{X}{H_i}}\, d\mu_{V_i} 
		+ \int_K \abs{\esc{\hat{\eta}_i}{X}}\,d\hat{\beta}_i  \\
		&\leq c(n,m,c_K,d_K,g_K,\Gamma \cap K) \sup_K \enorm{X}, \nonumber
	\end{align}
	for any $X \in C_c^1(U;\Rnm)$ such that $\spt X \subset K \Subset U$. 

    For any $x \in U$ such that $\Tan(\mu_i,x)$ exists, we denote by $\sqrt{\det g_S(\mu_i,x)}$ and $\sqrt{\det g_E(\mu_i,x) }$ the induced volume elements of the canonical pseudo-Euclidean metric of $\Rnm$ and the canonical Euclidean metric of $\R^{n+m}$ on $\Tan(\mu_i,x)$ respectively. Let
	\begin{equation}\label{eq:DefinitionS}
	S_i(x) = S(\Tan(\mu_i,x)) = \frac{\sqrt{\det g_S(\mu_i,x)}}{\sqrt{\det g_E(\mu_i,x) }} \in (0,1],
	\end{equation}
    where $\Tan(\mu_i,x)$ exists, that is, $S_i$ is defined for $\mu_i$ almost every $x \in U$. By \eqref{eq:UniformSpacelike}, we have that \begin{equation}\label{eq:UniformSpacelikeII}
		s_K := \inf_i \inf_{x \in K} S(\Tan(\mu_i,x)) > 0
	\end{equation}
	for every $K \Subset U$. For each $i \in \mathbf N$, here exists a spacelike
	$n$-rectifiable subset $M$ of $U$ and a multiplicity function $\theta_i \in L^1(M_i,\sigma^n)$ that takes values in the nonnegative integers such that
	\[
		\mu_i = \theta_i \sigma^n \resmes M_i.
	\]
	Then, define 
	\[ 
	\tilde{V}_i = \theta_i \H^n \resmes M_i \otimes \delta_{\Tan(M_i,\cdot)}
	\]  
	and 
	\[
		\tilde{\mu}_i = \theta_i \H^n \resmes M_i.
	\]
	For any $K \Subset U$, 
	\[
		\tilde{\mu}_i(K) 
		= \int_{M_i \cap K} \theta_i(x)\, d\H^n(x) 
		= \int_{M_i \cap K} \frac{\theta_i(x)}{S(\Tan(M_i,x))}\, \sigma^n(x) 
		\leq \frac{1}{S_K}\mu_i(K),
	\]
	which implies
	\[
		\sup_i\tilde{\mu}_i(K) \leq \frac{c_K}{s_K}.
	\]
	On the other hand, the Euclidean first variation of $\tilde{V}_i$ is given by 
	\[
		\delta_E \tilde{V}_i(X) = \int \Div^E_{\Tan(M_i,x)}(X) \, d\tilde{\mu}_i(x),
	\] 
	for any $X \in C_c^1(U;\Rnm)$, where 
	\[
		\Div^E_{\Tan(M_i,x)}(X) = \sum_{j=1}^n g_E(\varsigma_j,\nabla_{\varsigma_j}X)
	\]
	and $\{\varsigma_1,\dots,\varsigma_n\}$ is a $g_E$-orthonormal base of $\Tan(M_i,x)$ for $x \in U$. Let $\{\tau_1,\cdots,\tau_n\}$ be a $\lnorm{\cdot}$-orthonormal base of $\Tan(M_i,x)$. Each $\varsigma_j$ can be written as
	\[ 
	\varsigma_j = \sum_{\alpha=1}^n \esc{\varsigma_j}{\tau_\alpha}\tau_\alpha. 
	\] 
	By \eqref{eq:GradientEstimateI}, 
	\[ 
	\esc{\varsigma_j}{\tau_\alpha}^2 \leq v^2(\Tan(M_i,x)). 
	\]
	Therefore, if $X \in C_c^1(U;\Rnm)$ such that $\spt X \subset K \Subset U$, by \eqref{eq:AllardHypothesis}, 
	\[
	\vert \delta_E \tilde{V}_i(X) \vert \leq c(n,m,c_K,d_K,g_K,\Gamma \cap K) \sup_K \enorm{X}.
	\]
	Allard's closure theorem for Riemannian integer $n$-rectifiable varifolds \cite[Theorem 42.7]{Simon83} implies that there exists a subsequence (denoted by the same index) such that $\tilde{V}_i$ converge to a Riemannian integer $n$-rectifiable varifold $\tilde{V}$. 
	 
	Let $\tilde{\mu} = \mu_{\tilde{V}}$. Since $\tilde{V}$ is a Riemannian integer $n$-rectifiable varifold, there exists an $n$-rectifiable subset $M$ of $U$ and a positive multiplicity function $\tilde{\theta} \in L^1_{\opn{loc}}(M,\H^n)$ that takes values in the nonnegative integers such that $\tilde{\mu} = \tilde{\theta}\H^n \resmes M$. We need  to show that $\Tan(\tilde{\mu},x) = \Tan(M,x)$ is spacelike whenever $\Tan(\tilde{\mu},x)$ exists. Fix $K \Subset U$ and let $x_0 \in K$. By the varifold convergence of $\tilde{V}_i$ to $\tilde{V}$, there exists a sequence $(x_i,\Tan(\tilde{\mu}_i,x_i)) \subset \spt \tilde{V}_i$ with $x_i \in K$, such that $(x_i,\Tan(\tilde{\mu}_i,x_i)) \to (x_0,\Tan(\tilde{\mu},x_0))$, and by the continuity of $S$, we obtain $S(\Tan(\tilde{\mu}_i,x_i)) \to S(\Tan(\tilde{\mu},x_0))$. Thus $S(\Tan(M,x)) \geq s_K$ for every $x \in \spt\tilde{\mu} \cap K$ such that $\Tan(\tilde{\mu},x)$ exists. 
	
	Since $\tilde{V}_i$ converges to $\tilde{V}$ and $S$ is a continuous function on $G(n,m)$, 
	\begin{multline}
		\int_{M_i} \theta_i(x)f(x,\Tan(M_i,x))\,d\sigma^n(x)
		= \int_{M_i} \theta_i(x)S(\Tan(M_i,x))f(x,\Tan(M_i,x))\,d\H^n(x)\\
		\rightarrow \int_{M} \theta(x)S(\Tan(M,x))f(x,\Tan(M,x))\,d\H^n(x)
		=\int_{M} \theta(x)f(x,\Tan(M,x))\,d\sigma^n(x).
	\end{multline}
	That is, $\mu = \theta \sigma^n \resmes M$ and $\opn{Var}(\mu_i)$ converges to $\opn{Var}(\mu)$. In particular, the limit $\opn{Var}(\mu)$ does not depend on the choice of subsequence. Thus the original sequence $\opn{Var}(\mu_i)$ converges to $\opn{Var}(\mu)$, which proves assertion (2).
	
	By \eqref{eq:UniformMeanCurvature} and \eqref{eq:UniformSpacelike} we have that
	\[
		\sup_i \int_K \enorm{H_i}^2\, d\mu_i \leq m(n+m)g_Kd_k.
	\]	
	for any $K \Subset U$. From \cite[Theorem 3.1]{White21}, after passing to a subsequence, there exists a $\mu$-measurable vectorfield $\tilde{H}$ such that 
	\begin{equation}\label{eq:MeanCurvatureConvergence}
		\int \esc{X(x,\Tan(\mu_i,x))}{H_i(x)}\, d\mu_i(x) \rightarrow \int \esc{X(x,\Tan(\mu,x))}{\tilde{H}(x)}\, d\mu(x)
	\end{equation}
	for every $X \in \mathfrak{X}_n(U)$. Furthermore, by \cite[Corollary 3.2]{White21}, $H_i(x) \in \Tan(\mu_i,x)^\perp$ for $\mu_i$ almost every $x \in U$ implies $\tilde{H}(x) \in \Tan(\mu,x)^\perp$ for $\mu$ almost every $x \in U$.
	
	By $\eqref{eq:UniformConormal}$ and the convergence of $\Gamma_i$ to $\Gamma$ in $C^\infty$,
	\[ 
	\sup_i\esssup_{\Gamma \cap K} \enorm{\nu_i} \leq c(K \cap \Gamma).
	\]
	for any $K \Subset U$. From \cite[Theorem 3.3]{White21} there exists a $\Gamma$-measurable vectorfield $\tilde{\nu}$ such that 
	\begin{equation}\label{eq:CoNormalConvergence}
		\int \esc{Z(x,\Tan(\Gamma_i,x))}{\nu_i(x)}\, d\Gamma_i(x) \rightarrow \int \esc{Z(x,\Tan(\Gamma,x))}{\tilde{\nu}(x)}\, d\Gamma(x)
	\end{equation}
	for every $Z \in \mathfrak{X}_{n-1}(U)$.  
    
	Fix $K \Subset U$ such that $K \cap \Gamma \neq \emptyset$. Since $\Gamma_i \to \Gamma$ in $C^\infty$, there exists $i_0 \in \mathbf{N}$ such that $K \cap \Gamma_i \neq \emptyset$ for $i \geq i_0$, and a diffeomorphism $\phi_i : \Gamma_i \cap K \to \Gamma \cap K$ such that $\phi_i$ converges to the identity on $\Gamma \cap K$ and $\Tan^\perp(\Gamma_i,\phi(x))$ converges to $\Tan^\perp(\Gamma,x)$ smoothly. Let $\pi_i : \Tan(\Gamma_i,\phi_i(x))^\perp \to \Tan(\Gamma,x)^\perp$ be the orthogonal projection onto $\Tan(\Gamma,x)^\perp$ and $\tilde{\nu}_i(x) = \pi_i(\nu_i(\phi_i(x)))$. Then $\tilde{\nu}_i(x) \in \Tan(\Gamma,x)^\perp$ for $\sigma^{n-1} \resmes \Gamma$ almost every $x \in \Gamma$, so we can  smoothly identify $\Gamma_i$ with $\Gamma$ for $i$ sufficiently large. 
	
	Let $\delta > 0$. For each $i \geq i_0$ we can apply Lusin's theorem to the measurable maps 
	\[ \tilde{\nu}_i : \Gamma \cap K \to \Rnm \]
	that is, there exists a closed subset $F_i \subset K \cap \Gamma$ such that 
	\[ 
		\sigma^{n-1}((K \cap \Gamma) \setminus F_i) \leq \delta 2^{-i} 
	\]
	and such that $\tilde{\nu}_i$ is continuous in $F_i$ in the subspace topology. If we let $F = \cap_{i \geq i_0} F_i$, then 
	\[ 
		\sigma^{n-1}((K \cap \Gamma) \setminus F) \leq \sum_{i \geq i_0} \delta 2^{-i} \leq \delta
	\]
	and $\tilde{\nu}_i$ is continuous in $F$ for every $i \geq i_0$ in the subspace topology. 
	
	Since $\Gamma$ is $(n-1)$-dimensional and spacelike, $\Tan(\Gamma,x)^\perp$ has signature $(1,m)$ and $\{u \in \Tan(\Gamma,x)^\perp : \lnorm{u} \geq n_K > 0\}$ consist of two disjoint, closed convex sheets for every $x \in \Tan(\Gamma,x)^\perp$. Let $(U_\ell)_{\ell \in \mathbf{N}}$ be the connected components of $F$ such that $\sigma^{n-1}(U_\ell) > 0$. For each $i \geq i_0$, since $U_\ell$ is connected and $\tilde{\nu}_i|_{U_\ell}$ is continuous, $\tilde{\nu}_i(U_\ell) = \{0\}$ or $\tilde{\nu}_i(U_\ell)$ lies in one fixed sheet. After passing to a subsequence (labeled by the same index), we can assume that, for each $\ell$, $\tilde{\nu}_i(U_\ell) = \{0\}$ or $\tilde{\nu}_i(U_\ell)$ lies in one fixed spacelike cone for all $i$.  
	
	Suppose there exists $A \subset \{x \in \Gamma : \lnorm{\tilde{\nu}(x)} \leq 0, \, \tilde{\nu}(x) \neq 0 \}$ a Borel subset such that $\sigma^{n-1}(A) > 0$. Then $\sigma^{n-1}(A \cap K) > 0$ for some $K \Subset U$. Let $0 < \delta < \sigma^{n-1}(A \cap K)$. By Lusin's theorem, there exists a closed subset $F_\delta \subset \Gamma \cap K$ such that $\sigma^{n-1}((K \cap \Gamma) \setminus F_\delta) \leq \delta$ and such that, if $(U_\ell)_{\ell \in \mathbf{N}}$ are the connected components of $F_\delta$ with $\sigma^{n-1}(U_\ell) > 0$, after passing to a subsequence, we can assume that, for each $\ell$, $\tilde{\nu}_i(U_\ell) = \{0\}$ or $\tilde{\nu}_i(U_\ell)$ lies in one fixed sheet for all $i$. By \eqref{eq:CoNormalConvergence}, 
	\[
	\int \esc{Z}{\tilde{\nu_i}}\, d\Gamma \rightarrow \int \esc{Z}{\tilde{\nu}}\, d\Gamma
	\]
	for every $Z \in C_c^1(U;\Rnm)$ such that $\spt Z \subset K \Subset U$. By Mazur's lemma, there exits a sequence $(u_m)$ made of  finite convex combinations of elements $\tilde{\nu}_i$ that converges to $\tilde{\nu}$ strongly in $\mcal{L}^2(K,d\Gamma)$ and, after passing to a subsequence, $u_m(x) \to \tilde{\nu}(x)$ for $\sigma^{n-1} \resmes \Gamma$ almost every $x \in \Gamma \cap K$. Therefore, \textbf{\textbf{}}by the convexity of the sheets, we obtain $\tilde{\nu}(x) = 0$ or $\lnorm{\tilde{\nu}(x)} \geq n_K > 0$ for $x \in F_\delta$, but $A \cap K \subset (K \cap \Gamma) \setminus F_\delta$ and $\sigma^{n-1}(A \cap K) \leq \delta$, reaching a contradiction. 
	
	By \eqref{eq:CoNormalConvergence} and Lemma \ref{lemma:ConormalEquivalence}, for every $Z \in \mathfrak{X}_{n-1}(U)$ such that $0 < \lnorm{Z(x,\Tan(\Gamma_i,x))} \leq 1$ if $Z(x,\Tan(\Gamma_i,x)) \neq 0$, we get
	\[
	\int \esc{Z}{\nu_i}\, d\Gamma_i \leq \int \sqrt{\lnorm{Z}}\, d\Gamma_i \leq \Gamma_i(\spt Z)
	\]
	Letting $i \to \infty$, 
	\[
	\int \esc{Z}{\tilde{\nu}}\, d\Gamma \leq  \Gamma(\spt Z),
	\]
	from which it follows $\lnorm{\tilde{\nu}(x)} \leq 1$ for $\sigma^{n-1}$ almost every $x \in \Gamma$.
	
	For every $X \in C^1_c(U;\Rnm)$,
	\[
	\delta \opn{Var}(\mu_i)(X) = -\int \esc{H_i}{X}\, d\mu_i + \int_\Gamma \esc{\nu_i}{X}\, d\Gamma.
	\]
	By the varifold convergence of the $\opn{Var}(\mu_i)$ to $V$, and by \eqref{eq:MeanCurvatureConvergence} and \eqref{eq:CoNormalConvergence}, it follows that
	\[
	\delta V(X) = -\int \esc{\tilde{H}}{X}\, d\mu + \int_\Gamma \esc{\tilde{\nu}}{X}\, d\Gamma. 
	\] 
	Consequently, $H(\mu,\cdot) = H(\cdot)$, $\nu(\mu,\cdot) = \tilde{\nu}$ and $\mu \in S\mathscr{V}_n(U,\Gamma)$. We passed to a subsequence but, since the limit $H(\mu,\cdot)$ is independent of the choice of subsequence because $\delta \opn{Var}(\mu_i)(X)$ is bounded and every convergent subsequence has the same limit, in fact \eqref{eq:MeanCurvatureConvergence} and \eqref{eq:CoNormalConvergence} hold for the original sequence. 
\end{proof} 

\section{Spacelike Brakke flows with boundary}

Suppose that we have a smooth, properly embedded, solution $\{M_t\}_{t \in [0, T)}$ of spacelike mean curvature flow with fixed boundary in an open subset $U$ of $\Rnm$. According to the evolution equations (see \cite{LiSalavessa11}), we obtain
\begin{align*}
	\frac{d}{dt}\int_{M_t}\phi \, d\mu_{M_t} &= \int_{M_t} \left( \frac{d}{dt}\phi\right)d\mu_{M_t} + \phi \frac{d}{dt} \left(d\mu_{M_t}\right)\\
	&=\int_{M_t} (d\phi(H)+\partial_t\phi)d\mu_{M_t} - \phi \esc{H}{H}d\mu_{M_t}\\
	&=\int_{M_t} (-\phi \esc{H}{H} + \esc{\nabla\phi}{H} + \partial_t\phi)\, d\mu_{M_t}
\end{align*}
for all $\phi \in C_c^2(U \times [0,T);\mathbf{R}^{+})$. 

\begin{definition}
	An $n$-dimensional spacelike Brakke flow with boundary in an open subset $U$ of $\Rnm$ is a pair $((\mu_t)_{t \in I},\Gamma)$ where $\Gamma$ is a smooth, properly embedded $(n-1)$-dimensional spacelike submanifold of $U$ and where $(\mu_t)_{t \in I}$ is a family of Radon measures over an interval $I \subset \R$ such that:
	\begin{enumerate}[(1)]
		\item For almost every $t \in I$, $\mu_t \in S\mathscr{V}_n(U,\Gamma)$. 
		\item If $[a,b] \subset I$ and if $K \subset U$ is compact, then 
	   \begin{equation}\label{eq:TonegawaUniformAreaBound}
			\int_a^b \int_K \left( 1 - \lnorm{H} \right)\,d\mu_t dt < \infty. 
		\end{equation}
		\item If $[a,b] \subset I$ and $\phi \in C_c^2(U \times [a,b];\mathbf{R}^{+})$, then
		\begin{align*}
			\mu_b(\phi) - \mu_a(\phi)
			&\geq \int_a^b\int \left(-\phi \lnorm{H} + \esc{\nabla \phi}{H} + \frac{\partial \phi}{\partial t}\right) \, d\mu_t dt.
		\end{align*}
	\end{enumerate}
\end{definition} 

We denote the gradient function on the flowing varifold as follows. Let 
\[ v^2(\mu_t,x) = v^2(\Tan(\mu_t,x)) \]
where $\mu_t \in S\mathscr{V}_n(U,\Gamma)$ and $\Tan(\mu_t,x)$ exists, and let $v^2(\mu_t,x) = m$ otherwise.  

\subsection{Basic properties of Brakke Flow}

In this section, we present some basic properties of the the spacelike Brakke flow in pseudo-Euclidean space, analogous to the Euclidean case. 

\begin{theorem}\label{lemma:monotonicity}
	Let $(\mu_t)_{t \in I}$ be a spacelike Brakke flow in $U$ with boundary $\Gamma$. Suppose that, for $\phi \in C_c^2(U; \mathbf{R}^{+})$ and $[a,b] \subset I$, 
	\[ 
	K(v^2,[a,b]) := \sup_{t \in [a,b]}\int_{\spt\phi}v^2(\mu_t,\cdot)\,d\mu_t < \infty. 
	\]
	Then 
    \[
    \frac{1}{2} \int_a^b\int -\phi \lnorm{H}\, d\mu_tdt 
    \leq \mu_b(\phi) - \mu_a(\phi) + m (b-a)\max \vert \nabla^2\phi \vert_{\R^{n+m}} K(v^2,[a,b]).
    \]  
\end{theorem}

\begin{proof}
    Consider the $1$-covariant tensor $T \colon \Rnm \to \R$ given by 
    \[
        T(v)=\esc{\nabla\phi}{v}.
    \]
    By \eqref{eq:GradientEstimateI} and \cite[Lemma 6.6]{Ilmanen94},
	\begin{equation}\label{eq:GradientEstimateII}
		-\esc{\nabla^\perp\phi}{\nabla^\perp\phi} 
        \leq 	nv^2\abs{T}_{\R^{n+m}}^2 
        = mv^2 \abs{\nabla \phi}_{\R^{n+m}}^2 
        \leq 2mv^2\phi\max \vert \nabla^2\phi \vert_{\R^{n+m}},
    \end{equation}
    where $\vert \nabla^2 \phi \vert_{\R^{n+m}}$ is the norm of the Hessian of $\phi$ as a tensor on $\R^{n+m}$. Thus whenever $\phi > 0$,   
	\[
		0 \geq \esc{\left(\frac{\nabla\phi}{\sqrt{\phi}} - H\sqrt{\phi}\right)^{\perp}}{\left(\frac{\nabla\phi}{\sqrt{\phi}} - H\sqrt{\phi}\right)^{\perp}}
		= \frac{\esc{\nabla^\perp\phi}{\nabla^\perp\phi}}{\phi} - 2\esc{\nabla\phi}{H} + \phi\esc{H}{H}
	\]
	Therefore,
	\begin{equation}\label{Eq:Young'sInequality}
        \langle \nabla\phi, H\rangle - \phi \esc{H}{H} 
        \geq \frac{\esc{\nabla^\perp\phi}{\nabla^\perp\phi}}{2\phi} - \frac{\phi}{2}\esc{H}{H}.	
	\end{equation}
	By the definition of the spacelike Brakke flow,
	\begin{align*}
		\mu_b(\phi)-\mu_a(\phi) 
        &\geq 
        \int_a^b\int \left( -\phi\lnorm{H} + \esc{\nabla\phi}{H} \right) \, d\mu_tdt\\ 
		&\geq 
		\frac{1}{2}\int_a^b\int -\phi \lnorm{H} \, d\mu_tdt -m \max \vert \nabla^2\phi \vert_{\R^{n+m}} \int_a^b\int v^2 \,d\mu_tdt \\
        &\geq 
		\frac{1}{2}\int_a^b\int -\phi \lnorm{H} \, d\mu_tdt -m(b-a) \max \vert \nabla^2\phi \vert_{\R^{n+m}} K(v^2,[a,b]).
	\end{align*}
\end{proof}

\begin{corollary} If $K(v^2,I) < \infty$, then 
    \[
        t \in I \mapsto \mu_t(\phi)+ m(\max \vert \nabla^2 \phi \vert_{\R^{n+m}} )K(v^2,I) t
    \]
    is a non-decreasing function of $t$. 
\end{corollary}

\begin{corollary}[Increasing property of mass]
	Let $(\mu_t)_{t \in I}$ be a spacelike Brakke flow in $\Rnm$ with boundary $\Gamma$. Suppose that, for $[a,b] \subset I$,
	\[ 
	\int_a^b\int v^2\,d\mu_tdt < \infty. 
	\]
	Then 
	\[ \int d\mu_b \geq \int d\mu_a. \]
\end{corollary}

\begin{proof}
    Let $\phi = \phi_R \in C_c^2(U;[0,1])$ satisfy $\phi = 1$ on $B_R(0)$, $\phi = 0$ off $B_{2R}(0)$, and $\abs{\nabla \phi}_{\R^{n+m}} \leq 2/R$. Then by \eqref{eq:GradientEstimateII} and \eqref{Eq:Young'sInequality}, 
    \begin{align*}
        \int \phi^2\, d\mu_b - \int  \phi^2 \, d\mu_a 
        &\geq \int_a^b\int 2 \lnorm{\nabla^\perp\phi} \,d\mu_tdt\\
        &\geq \int_a^b\int -2mv^2 \vert \nabla\phi \vert_{\R^{n+m}}\,d\mu_tdt\\
        &\geq -\frac{8m}{R^2} \int_a^b\int v^2\,d\mu_tdt
    \end{align*}
    Letting $R \to \infty$, we obtain the desired result.  
\end{proof}

\subsection{Compactness Property for the Lorentzian Brakke Flow}

In the following theorem, we write $H_i(x,t)$ for $H(\mu_t^i,x)$, $v^2_i(x,t)$ for $v^2(\mu_t^i,x)$ and $\nu_i(x,t)$ for $\nu(\mu_t^i,x)$. 

\begin{theorem}\label{th:CompactnessSpacelikeBrakkeFlows}
	For each $i \in \mathbf{N}$, let $t \in [a,b] \mapsto \mu_t^i$ be an $n$-dimensional spacelike Brakke flow in $U$ with boundary $\Gamma_i$, where $\Gamma_i$ are smooth spacelike $(n-1)$-dimensional submanifold of $U$ that converge in $C^\infty$ to a smooth spacelike submanifold $\Gamma$. Suppose
	\begin{equation}\label{eq:AreaBound} 
		c_K = \sup_i \sup_{t \in [a,b]} \mu_t^i(K) < \infty,
	\end{equation}
	\begin{equation}\label{eq:GradientBound} 
		g_K = \sup_i \sup_{t \in [a,b]} \sup \{ v_i^2(x,t) : x \in K\} < \infty, 
	\end{equation}
    and
	\[ 
		n_K =  \inf_i \inf_{t \in [a,b]} \inf \{\lnorm{\nu_i(x,t)} : x \in K \cap \Gamma_i, \, \nu_i(x,t) \neq 0 \}  > 0 
	\]
	for every $K \Subset U$. Then, after passing to a subsequence,
	\begin{enumerate}[(1)]
		\item $\mu_t^i$ converges to a Radon measure $\mu_t$ for every $t \in [a,b]$.
		\item $t \in [a,b] \mapsto \mu_t$ is an $n$-dimensional spacelike Brakke flow in $U$ with boundary $\Gamma$.
		\item For almost every $t \in [a,b]$, after passing to a further subsequence (which may depend on $t$), 
        \begin{equation}\label{eq:VarifoldConvergence}
            \opn{Var}(\mu_t^i) \to \opn{Var}(\mu_t),
        \end{equation}
        \begin{equation}\label{eq:MeanCurvatureConvergenceII}
            \int \esc{H_i(x,t)}{X(x,\Tan(\mu_t^i,x))}\, d\mu_t^i(x) \rightarrow \int \esc{H(x,t)}{X(x,\Tan(\mu_t,x))}\, d\mu_t(x),
        \end{equation}
        \begin{equation}\label{eq:ConormalConvergence}
            \int \esc{\nu_i(x,t)}{Y(x,\Tan(\Gamma_i,x))}\, d\Gamma_i(x) \rightarrow \int \esc{\nu(x,t)}{Y(x,\Tan(\Gamma,x))}\, d\Gamma(x)
        \end{equation}
        for every $X \in \mathfrak{X}_n(U)$ and $Y \in \mathfrak{X}_{n-1}(U)$.
	\end{enumerate} 
\end{theorem}

\begin{proof}
    Let $\mathcal{F}$ be a countable subset of $C_c^2(U;\R^+)$ such that the linear span of $\mathcal{F}$ is dense in $C_c(U;\R^+)$. By \eqref{eq:AreaBound}, \eqref{eq:GradientBound} and Proposition \ref{lemma:monotonicity}, for each $\phi \in \mathcal{F}$, the function
    \begin{equation}\label{CompactnessTheoremFunctions}
    t \in [a,b] \mapsto \mu_t^i(\phi) + c(\phi)t
    \end{equation}
    is monotone non-decreasing function. Moreover, each such function is uniformly bounded for all $i \in \mathbf N$. Therefore, by Helly's selection theorem, there exists a subsequence (denoted by the same index) such that, for each fixed $\phi \in \mathcal F$, the function \eqref{CompactnessTheoremFunctions} converges pointwise to a limit function. Since $\mathcal{F}$ is countable, extracting a diagonal subsequence (again, denoted by the same index), we can assume that each of the functions \eqref{CompactnessTheoremFunctions} converges to a limit function. By the Riesz Representation Theorem, for each $t \in [a,b]$, $\mu_t^i$ converges to a Radon measure $\mu_t$.   

	For each $\phi \in C_c^2(U;\R^+)$, by Theorem \ref{lemma:monotonicity},
	\begin{equation}\label{BoundMeanCurvature}
		\sup_i \int_a^b\int -\phi \lnorm{H_i}\,d\mu_t^idt < \infty.
	\end{equation}
	By Fatou's Lemma, 
	\[
		\int_a^b \left( \liminf_i \int -\phi \esc{H_i}{H_i}\, d\mu_t^i \right)dt < \infty 
	\]
	so for almost every $t$, 
	\begin{equation}\label{FatousLemma}
		\liminf_i \int -\phi \esc{H_i}{H_i}\, d\mu_t^i < \infty.
	\end{equation}
	For every such $t$, there is a subsequence (denoted by the same index) such that
	\[
		\sup_i \int -\phi \esc{H_i}{H_i}\, d\mu_t^i < \infty.
	\]
	By Theorem \ref{CTRV}, $\mu_t \in S\mathscr{V}_n(U,\Gamma)$ and \eqref{eq:VarifoldConvergence}, \eqref{eq:MeanCurvatureConvergenceII} and \eqref{eq:ConormalConvergence} hold. 
    
	We notice that $\opn{Var}(\mu_t)$ is well defined independent of the subsequence depending on $t$ for almost every $t \in [a,b]$. This is because an integral varifold $V$ is uniquely determined by its associate measure $\mu_V$. However, we emphasize that the convergence of $\opn{Var}(\mu_t^i)$ to $\opn{Var}(\mu_t)$ as varifolds requires extracting a subsequence depending on $t$, as we have done above.
	
	It remains only to show that $t \in [a,b] \mapsto \mu_t$ is an spacelike $k$-dimensional Brakke flow with boundary $\Gamma$. If $[c,d] \subset [a,b]$ and $\varphi \in C^2_c(U \times [c,d];\R^{+})$, then 
	\begin{align*}
		\mu_d^i(\varphi) - \mu_c^i(\varphi) - \int_c^d\int \frac{\partial \varphi}{\partial t}\, d\mu_t^idt 
		\geq \int_c^d \int -\varphi \esc{H_i}{H_i} + \esc{\nabla\varphi}{H_i}\, d\mu_t^idt
	\end{align*}
	By \eqref{Eq:Young'sInequality}, 
	\[ 
		\int -\varphi \esc{H_i}{H_i} + \esc{\nabla\varphi}{H_i}\,d\mu_t^i \geq \int_{\{\varphi > 0\}} -mv^2\frac{\enorm{\nabla \varphi}^2}{2\varphi}\, d\mu_t^i \geq -C,
	\]
	where $C:=C(c_{\spt\varphi},g_{\spt\varphi},\max \vert \nabla^2\varphi \vert_{\R^{n+m}}) > 0$ is bounded independently of $i$ and $t$. By Fatou's Lemma,
	\begin{gather*}
		\liminf_i \int_c^d \left( C + \int -\varphi \esc{H_i}{H_i} + \esc{\nabla\varphi}{H_i}\,d\mu_t^i  \right)dt\\ 
		\geq \int_c^d \liminf_i \left( C + \int -\varphi \esc{H_i}{H_i} + \esc{\nabla\varphi}{H_i}\,d\mu_t^i  \right)dt,
	\end{gather*}
	which implies
	\begin{gather*}
		\liminf_i \int_c^d \left(\int -\varphi \esc{H_i}{H_i} + \esc{\nabla\varphi}{H_i}\,d\mu_t^i  \right)dt\\ 
		\geq \int_c^d \liminf_i \left(\int -\varphi \esc{H_i}{H_i} + \esc{\nabla\varphi}{H_i}\,d\mu_t^i  \right)dt.
	\end{gather*}
	Letting $i \to \infty$ gives
	\begin{equation}\label{eq:compactness1}
		\mu_d(\varphi) - \mu_c(\varphi) - \int_c^d\int \frac{\partial \varphi}{\partial t}\, d\mu_tdt 
		\geq \int_c^d \lambda(t) dt,
	\end{equation}
	where 
	\[ 
		\lambda(t):= \liminf_i \int -\varphi \esc{H_i}{H_i} + \esc{\nabla\varphi}{H_i}\, d\mu_t^i. 
	\]
	For each $t$ with $\lambda(t) < \infty$, there is a subsequence (denoted by the same index) such that
	\[ 
		\int -\varphi \lnorm{H_i} + \esc{\nabla\varphi}{H_i}\, d\mu_t^i \rightarrow \lambda(t). 
	\]
	For such $t$, we have (as above), $\opn{Var}(\mu_t^i) \to \opn{Var}(\mu_t)$ and 
	\[
		\int \esc{H_i(x,t)}{X(x,\Tan(\mu_t^i,x))} \, d\mu_t^i(x) \rightarrow \int  \esc{H(x,t)}{X(x,\Tan(\mu_t,x)} \, d\mu_t(x)
	\] 
	for all $X \in \mathfrak{X}_n(U)$, which implies 
	\begin{equation}\label{eq:compactness2}
		\int \esc{\nabla\varphi}{H_i}\, d\mu_t^i \rightarrow \int \esc{\nabla\varphi}{H}\, d\mu_t
	\end{equation}
	For all $x \in \spt\mu_t$ where $\Tan(\mu_t,x)$ exists, $-\esc{\cdot}{\cdot}$ is a positive definite metric on $\Tan(\mu_t,x)^\perp$. If $Y \in \mathcal{L}^2_{\opn{loc}}(U;\Tan(\mu_t,\cdot)^\perp)$ such that $\int -\esc{Y}{Y}\,d\mu_t = 1$, then 
	\begin{align*}
		 \int -\varphi^{1/2} \esc{H}{Y}\, d\mu_t  
		 &\leq \int \varphi^{1/2} \sqrt{-\lnorm{H}}\sqrt{-\lnorm{Y}}\, d\mu_t\\
		 &\leq \left( \int -\varphi \esc{H}{H}\, d\mu_t \right)^{1/2}.
	\end{align*}
	Now, if $\int -\varphi \lnorm{H}\, d\mu_t \neq 0$, let 
	\[
		Y = \frac{\varphi^{1/2}H}{\left( \int -\varphi \lnorm{H}\, d\mu_t \right) ^{1/2}} \in \mathcal{L}^2_{\opn{loc}}(U;\Tan(\mu_t,\cdot)^\perp). 
	\] 
	Then 
	\[
		\int -\varphi^{1/2} \esc{H}{Y} \, d\mu_t = \left( \int -\varphi \esc{H}{H}\, d\mu_t \right)^{1/2}.	
	\] 
	Since $C_c^1(U;\Tan(\mu_t,\cdot)^\perp)$ is dense in $\mathcal{L}^2_{\opn{loc}}(U;\Tan(\mu_t,\cdot)^\perp)$ by \cite[\S7.4]{Ilmanen94}, we can compute the following $\mathcal{L}^2$-norm by duality as in \cite[Page 37]{Ilmanen94}, that is, 	
	\begin{equation}\label{eq:NormDuality}
		\begin{aligned}
		\left( \int -\varphi \lnorm{H}\, d\mu_t \right)^{1/2} \\
		&= \sup \left\{	\int -\varphi^{1/2} \esc{H}{Y} \, d\mu_t: Y \in C^1_c(U;\Tan(\mu_t,\cdot)^\perp), \int -\lnorm{Y}\,d\mu_t = 1 \right\} 
		\end{aligned}
	\end{equation}
	Thus
	\begin{align*}
		\int -\varphi^{1/2} \esc{H}{Y} \, d\mu_t &= \lim_{i \to \infty} \int -\varphi^{1/2} \esc{H_i}{Y} \, d\mu_t^i\\
		&\leq \liminf_i \left( \int -\varphi \lnorm{H_i}\, d\mu_t^i \right)^{1/2} \left( \int -\lnorm{Y}\,d\mu_t \right)^{1/2},
	\end{align*}
	from which it follows that
	\begin{equation}\label{eq:compactness3}
		\int -\varphi \lnorm{H}\, d\mu_t 
		\leq \liminf_i \int -\varphi \lnorm{H_i}\, d\mu_t^i.
	\end{equation}
	Finally, using \eqref{eq:compactness1}, \eqref{eq:compactness2} and \eqref{eq:compactness3} we obtain
	\[
	\mu_d(\varphi) - \mu_c(\varphi) 
	\geq \int_c^d \left( \int -\varphi \lnorm{H} + \esc{\nabla\varphi}{H} + \frac{\partial \varphi}{\partial t} \right)\, d\mu_t^idt. \qedhere
	\]  
\end{proof}

\section{Existence of spacelike Brakke flow in $\R^{n,m}$}\label{Section:EllipticRegularization}

In this section, we adapt Ilmanen's elliptic regularization construction \cite{Ilmanen94} to the pseudo-Euclidean setting. The main result of this section is the following.

\begin{theorem}[Ilmanen's Elliptic Regularization] Let $\Gamma$ be a smooth, properly embedded, $(n-1)$-dimensional, spacelike submanifold of $\Rnm$, and let $M_0$ be a spacelike $n$-rectifiable subset of $\Rnm$. Suppose that 
\begin{enumerate}[(1)]
    \item $\partial M_0 = \Gamma$.
    \item For every $K \Subset \Rnm$,
        \begin{equation}\label{eq:InitialGradientBound}
            s_K = \inf_{x \in K} S(\Tan(M_0,x)) > 0.
        \end{equation}
    \item $M_0 = \opn{graph}(u_{M_0})$ for some Lipschitz function $u_{M_0} : \Omega \Subset \R^n \to \R^m$.
\end{enumerate}
Then there exists an $n$-dimensional spacelike Brakke flow
\[
    t \in [0,\infty) \mapsto \mu_t
\]
with boundary $\Gamma$ such that 
\[
\mu_t \rightharpoonup \sigma^n \resmes M_0
\]
as $t \to 0$. 
\end{theorem}

\begin{proof}
Let $\Gamma^{*} = (M_0 \times \{0\}) \cup (\Gamma \times [0,\infty))$ and $M_0^* = (M_0 \cup \Gamma) \times [0,\infty)$. Then there exists a Lipschitz function $u_{M_0^*} : \Omega \times [0,\infty) \subset \R^{n+1} \to \R^m$ such that $M_0^* = \opn{graph}(u_{M_0^*})$ and $\partial M_0^* = \Gamma^*$. By \eqref{eq:InitialGradientBound}, 
\[
s_{K \times (a,b)} = \inf_{x \in K \times (a,b)} S(\Tan(M_0^*,x)) > 0 
\]
for any $K \times (a,b) \Subset \RnmR$. 

Let $\mathcal{G}$ be the family of spacelike $(n+1)$-rectifiable subsets $\Sigma$ of $\RnmR$ satisfying:
\begin{enumerate}[(1)]
	\item $\partial \Sigma = \Gamma^*$.
	\item There exists a Lipschitz function $u_\Sigma : \Omega \times [0,\infty) \subset \R^{n+1} \to \R^m$ such that $\Sigma = \opn{graph}(u_\Sigma)$.
	\item $\inf_{x \in K \times (a,b)} S(\Tan(\Sigma,x)) \geq s_{K \times (a,b)}$, for any $K \times (a,b) \Subset \RnmR$.
\end{enumerate}
By construction, $M_0^* \in \mcal{G}$. 

For $\lambda > 0$, we would like to maximize Ilmanen's functional 
\[
I^\lambda(\Sigma) = \int_{(x,z) \in \RnmR} e^{-\lambda z}\, d\mu_\Sigma(x,z)
\]
for $\Sigma \in \mcal{G}$. Note that $I^\lambda$ is the area functional with respect to the Ilmanen metric $e^{-2\lambda z/(n+1)}(g_{\Rnm} + dz^2)$. To this end, let $(\Sigma_i)_{i \in \mathbf N} \subset \mathcal{G}$ such that 
\[ 
I^\varepsilon(\Sigma_i) \nearrow \sup_{P \in \mathcal{G}} I^\lambda(P).
\]
Since every $P  \in \mathcal{G}$ is a spacelike graph over $\Omega \times [0,\infty)$ and $\Omega \Subset \R^n$, we obtain
\begin{equation}\label{eq:FunctionalBound}
	I^\lambda(P) \leq \int_{\Omega \times [0,\infty)} e^{-\lambda z} < \infty. 
\end{equation}
Moreover, by mollification, as in \cite[Theorem 4.1]{BartnikSimon82}, we can assume that each $\Sigma_i$ is smooth. By Ascoli-Arzela, there exists $\Sigma \in \mcal{G}$ such that $u_{\Sigma_i} \to u_\Sigma$ strongly in $C^0(\Omega \times [0,\infty))$ and weakly in $H^1(\Omega \times [0,\infty))$. Now, by Minkowski determinant theorem, the area functional is concave as in \cite[Proposition 6.1]{Bartnik88} and, by \cite[Section 8.2.2]{Evans98},
\[ 
 \limsup_{i \to \infty}\opn{area}(\Sigma_i \cap A) 
 \leq \opn{area}(\Sigma \cap A).
\]
for any Borel subset $A$ of $\RnmR$. Therefore,
\[
\sup_{P \in \mathcal{G}} I^\lambda(P) 
= \limsup_{i \to \infty} I^\lambda (\Sigma_i) 
\leq I^\lambda (\Sigma),
\] 
that is, $\Sigma$ maximizes $I^\lambda$ in $\mcal{G}$. 

Suppose there exists a smooth, compactly supported vector field $Y$ on $\Rnm \times \R$ and a family of diffeomorphisms $(\Phi^s)_{s \in (-\delta,\delta)}$ with 
\[ 
\left.\frac{\partial}{\partial_s}\right|_{s=0} \Phi^s = Y, \quad \Phi^0 = id, \quad \Phi^s|_{(\Rnm \times \R) \setminus U} = id
\]
for some neighborhood $U \Subset (\Rnm \times \R) \setminus \Gamma^*$ sufficiently small, and such that
\[
\left. \frac{\partial}{\partial s} \right|_{s=0} I^\lambda(Q_s)  
> 0,
\]
where $Q_s$ is the pushforward of $\Sigma$ by the diffeomorphism $\Phi^s$. Then, $\{\Phi^s\}_{s \in (-\delta,\delta)}$ increases $I^\lambda$ in $U$. Since $\Sigma$ maximizes $I^\lambda$ in $\mathcal{G}$, there exists point $x_0 \in Q_{s_0} \cap U$, where $\Tan(Q_{s_0},x_0)$ exists, such that $S(\Tan(Q_{s_0},x_0)) < \inf_{x \in U} S(\Tan(\Sigma,x))$. This implies that
\[ 
	\opn{area}(Q_{s_0} \cap U) < \opn{area}(\Sigma \cap U). 
\]
However, since $Q_s$ increases $I^\lambda$ in $U$,   
\begin{gather*}
	\int_U e^{-\lambda z}\,d\mu_{Q_{s_0}} > \int_{U} e^{-\lambda z}\,d\mu_\Sigma, 
\end{gather*}
which implies
\[
	e^{\lambda h} \opn{area}(Q_{s_0} \cap U) > \opn{area}(\Sigma \cap U),
\]
where $h = (\max_U z - \min_U z)>0$. Since $U$ is arbitrarily small, we reach a contradiction letting $h \to 0$. Therefore, $\Sigma$ is stationary for $I^\lambda$ in $(\RnmR) \setminus \Gamma^*$, that is, 
\begin{equation}\label{eq:WeightedFirstVariation}
\delta^\lambda \opn{Var}(M)(X) 
= \int \Div^\lambda_{M}X \,d\mu_M 
= \int \Div_M(e^{-\lambda z}X) - \lambda e^{-\lambda z} \esc{\partial_z^\perp}{X}\, d\mu_M = 0, 
\end{equation}
where $\delta^\lambda\Var(M)$ is the first variation of $\Sigma$ with respect to the Ilmanen metric, and $X$ is any $C^1$, compactly supported vector field such that $\spt X \subset (\RnmR) \setminus \Gamma^*$. Then $\Sigma$ has locally bounded first variation with 
\[ 
H =  -\lambda \partial_z^\perp 
\]
for $\mu_\Sigma$ almost everywhere on $\RnmR$, and 
\[ 
\tilde{\beta}(\Sigma) \subset  \Gamma^*.
\]
Moreover, by Theorem \ref{CTRV},    
\[ 
\lnorm{\nu(\Sigma,x)} = 1 
\]
for $\Gamma^*$ almost every $x \in \RnmR$.

Let $\Sigma = \Sigma^\lambda$ the spacelike $(n+1)$-rectifiable subset of $\RnmR$ constructed above, and  let $\mu^\lambda$ its associated Radon measure. Since $\Sigma^\lambda$ is a Lipschitz graph over $\Omega \times [0,\infty) \subset \R^{n+1}$, the areas of $\Sigma^\lambda$ are locally uniformly bounded, independently of $\lambda > 0$. In order to prove that the flow cannot vanish immediately, we need to find uniform lower bounds on area as $\lambda \to \infty$.

\begin{claim}
	If $\xi \in C_c^{0,1}(\R;\R^+)$, then it is possible to insert $Y = \xi(z)\partial_z$ in \eqref{eq:WeightedFirstVariation}, that is, 
	\[ 
	0 = \int e^{-\lambda z}\left(\xi_z \langle \partial_z^\top, \partial_z \rangle -\lambda \xi\right)\,d\mu^\lambda.  
	\] 
\end{claim}

\begin{proof}
	First assume that $\xi \in C_c^1(\R;\R^+)$. Fix a point $x_0 \in \Rnm$ and let $\sigma_R \in C_c^1(\Rnm;\R^{+})$ satisfy $\sigma_R = 1$ on $B_R(x_0)$, $\sigma_R = 0$ off $B_{2R}(x_0)$, $0 \leq \sigma_R \leq 1$ and $ \abs{\nabla \sigma_R}_{\R^{n+m}} \leq 2/R$. Inserting $Y = \sigma_R \xi \partial_z$ in $\eqref{eq:WeightedFirstVariation}$, we obtain
	\begin{align*}
		0 &= \int e^{-\lambda z}\left(\Div_{\Sigma^\lambda} Y - \lambda \esc{\partial_z}{Y} \right)\,d\mu^\lambda\\
		&= \int e^{-\lambda z}\left( \xi\esc{ \nabla^\top\sigma_R}{\partial_z} + \sigma_R\xi_z \esc{\partial_z^\top}{\partial_z} - \lambda \sigma_R\xi \right)\,d\mu^\lambda.
	\end{align*} 
	Since the gradient function is bounded on every $K \Subset \Rnm \times \R$, there exists $c_1 = c_1(n,
	\spt\xi) > 0$ such that
	\[
	\abs{\int e^{-\lambda z}\xi\esc{\nabla^\top \sigma_R}{\partial_z} \,d\mu} \leq \frac{2c_1}{R}\max\abs{\xi} I^\lambda(\Sigma^\lambda),
	\]
	so the first term disappear as we let $R \to \infty$. For the rest of the integrand, there exists $c_2 = c_2(n,
	\spt\xi) > 0$ such that
	\[
	\abs{e^{-\lambda z}\left( \sigma_R\xi_z \esc{\partial_z^\top}{\partial_z} - \lambda \sigma_R\xi\right)}
	\leq 
	\parent{ \lambda^{-1} c_2 \max\abs{\xi_z} + \max\abs{\xi}} \lambda e^{-\lambda z} 
	\]
	a fixed function in $L^1(\mu^\lambda)$. Letting $R \to \infty$, we obtain by the dominated convergence theorem,
	\[ 
	0 = \int e^{-\lambda z}\parent{\xi_z \esc{\partial_z^\top}{\partial_z} -\lambda \xi}\,d\mu^\lambda. 
	\]
	
	Now, let $\xi \in C_c^{0,1}(\R;\R^+)$ and approximate $\xi$ by $\xi^i \in C_c^1(\R;\R^+)$ such that $\xi^i \leq \xi$, $\xi^i \to \xi$ uniformly and $\xi_z^i \rightharpoonup \xi_z$ weakly-$\ast$ in $L^\infty(\R)$. We can immediately pass to the limits in the second term of the integrand. As for the first term, by the co-area formula, 
	\begin{align*}
	\int e^{-\lambda z}\xi_z \esc{\partial_z^\top}{ \partial_z}\,d\mu^\lambda 
	&= 
	\int e^{-\lambda z}\xi_z \lnorm{\nabla^{\Sigma^\lambda}\zeta} \,d\mu^\lambda \\
	&= 
	\int_\R \xi_z \int_{\Sigma^\lambda_z} e^{-\lambda z} \sqrt{\lnorm{\nabla^{\Sigma^\lambda}\zeta}} 1_{\spt\xi} \,d\mu_{\Sigma^\lambda_z}dz.
	\end{align*}
	where $\zeta(x,z) = z$ and $\Sigma^\lambda_z = \partial(\Sigma^\lambda \resmes (\Rnm \times [z,\infty)))$. The function 
	\[
	z \mapsto \int_{\Sigma^\lambda_z} e^{-\lambda z} \sqrt{\lnorm{\nabla^{\Sigma^\lambda} \zeta}} 1_{\spt\xi} \,d\mu_{\Sigma^\lambda_z}
	\]
	is in $L^1(\R)$ since 
	\begin{multline*}
	\int_\R \parent{\int_{\Sigma^\lambda_z} e^{-\lambda z} \sqrt{\lnorm{\nabla^{\Sigma^\lambda} \zeta}} 1_{\spt\xi}\,d\mu_{\Sigma^\lambda_z}}dz \\
	= 
	\int_{\RnmR} e^{-\lambda z} \esc{\partial_z^\top}{\partial_z} 1_{\spt\xi}\,d\mu^\lambda 
	\leq c_2 I^\lambda(\Sigma^\lambda) < \infty, 
	\end{multline*}
	so we can pass to limits in the right-hand side as well.
\end{proof}

Replace $\xi$ by $\xi e^{\lambda z}$, to obtain 
\begin{equation}\label{eq:AreaEstimateI}
	0 
	=  \int  \parent{\xi_z + \lambda\xi} \esc{\partial_z^\top}{\partial_z} - \lambda \xi \,d\mu^\lambda  
	= \int \xi_z \esc{\partial_z^\top}{\partial_z} - \lambda \xi \esc{\partial_z^\perp}{\partial_z} \,d\mu^\lambda. \nonumber
\end{equation}
If we let $\xi$ defined by 
\[
\xi(z) = \begin{cases}
	0, & \text{ on $(-\infty,a]$}\\
	1, & \text{ on $[a+\delta,b]$}\\
	0, & \text{ on $[b+\delta,\infty)$},
\end{cases}
\]
and linearly interpolated between, then by \eqref{eq:AreaEstimateI} we obtain	
\[
	\frac{1}{\delta} \int_{\Sigma^\lambda(b,b+\delta)}\lnorm{ \partial_z^\top} \,d\mu^\lambda 
	-\frac{1}{\delta} \int_{\Sigma^\lambda(a,a+\delta)}\lnorm{\partial_z^\top} \,d\mu^\lambda 
	= -\int \lambda \xi \lnorm{ \partial_z^\perp} \,d\mu^\lambda,
\]
where $\Sigma^\lambda(a,b)$ denotes $\Sigma^\lambda \resmes (\Rkn \times (a,b))$. Thus, for any fixed $\delta>0$ the function
\[
	f_\delta(z) = \frac{1}{\delta} \int_{\Sigma^\lambda(z,z+\delta)} \lnorm{ \partial_z^\top }\,d\mu^\lambda 
\]
is an increasing function of $z \in [0,\infty)$. Let
\[ 
	f(z) 
	= \int_{\Sigma^\lambda_z} \sqrt{\lnorm{\partial_z^\top}}\,d\mu_{\Sigma^\lambda_z} 
	= \int_{\Sigma^\lambda_z} \sqrt{\lnorm{\nabla^{\Sigma^\lambda} \zeta}} \,d\mu_{\Sigma^\lambda_z}, 
\]
for $z \in Z_1$, where $Z_1$ is the full measure subset of $[0,\infty)$ on which $\Sigma^\lambda_z = \langle \Sigma^\lambda, \zeta, z \rangle$ and the other good properties of the slice lemma \cite[1.19]{Ilmanen94} hold. By the co-area formula, 
\[ 
	f_\delta(z) = \frac{1}{\delta} \int_z^{z+\delta} f(z')\, dz',
\]
and, by the Lebesgue differentiation theorem and the increasing property of $f_\delta$, there is full measure set $Z \subset Z_1$ such that $\lim_{\delta \to 0}f_\delta(z) = f(z)$ for $z \in Z$. Therefore, we can pass to limits in the previous to obtain
\begin{equation} \label{eq:IlmanenMonotonicity}
	\int_{\Sigma^\lambda_a} \sqrt{\lnorm{\partial_z^\top}} \, d\mu_{\Sigma^\lambda_a} 
	- \int_{\Sigma^\lambda(a,b)} \lambda \lnorm{\partial_z^\perp} \,d\mu^\lambda 
	= \int_{\Sigma^\lambda_b} \sqrt{\lnorm{\partial_z^\top}} \, d\mu_{\Sigma^\lambda_b}.
\end{equation}

\begin{claim}
	For almost every $a,b \in [0,\infty)$ with $a \leq b$, we have 
	\begin{equation}\label{eq:LoweBoundArea}
	\opn{area}(\Sigma^\lambda \resmes \{a<z<b\}) \geq (b-a+\lambda^{-1})\opn{area}(M_0).
	\end{equation}
\end{claim}

\begin{proof} By \eqref{eq:IlmanenMonotonicity}, 
	\[
	\int_{\Sigma^\lambda_a} \sqrt{\lnorm{\partial_z^\top}} \, d\mu_{\Sigma^\lambda_{a}} 
	- \int_{\Sigma^\lambda(a,b)} \lambda \lnorm{\partial_z^\perp} \,d\mu^\lambda 
	= \int_{\Sigma^\lambda_b} \sqrt{\lnorm{\partial_z^\top}}\, d\mu_{\Sigma^\lambda_b}.
	\]
	This implies that the function 
	\[
	f(z) = \int_{\Sigma^\lambda_z} \sqrt{\lnorm{\partial_z^\top}} \,d\mu_{\Sigma^\lambda_z}
	\]
	is an increasing function of $z \in [0,\infty)$. Then 
	\begin{align*}
		\int_{\Sigma^\lambda(a,b)}\, d\mu^\lambda &= \int_{\Sigma^\lambda(a,b)} \lnorm{\partial_z^\top} - \lnorm{\partial_z^\perp} \,d\mu^\lambda\\
		&\geq \int_{\Sigma^\lambda(a,b)} \sqrt{\lnorm{\partial_z^\top}} -  \lnorm{\partial_z^\perp} \,d\mu^\lambda\\
		&\geq \int_a^b\int_{\Sigma^\lambda_z} \sqrt{\lnorm{\partial_z^\top}} \,d\mu_{\Sigma^\lambda_z}dz + \lambda^{-1} \int_{\Sigma^\lambda_b} \sqrt{\lnorm{\partial_z^\top}}\,d\mu_{\Sigma^\lambda_b} \\
		&\geq (b-a + \lambda^{-1}) \int_{\Sigma^\lambda_0} \sqrt{\lnorm{\partial_z^\top}}\,d\mu_{\Sigma^\lambda_0}\\
        &\geq (b-a+\lambda^{-1})\opn{area}(M_0),      
	\end{align*}
	where we have used that $\Sigma^\lambda_0 = M_0$ and $\lnorm{\partial_z^\top} \geq 1$. 
\end{proof}

For $t \geq 0$, let $\tilde{\mu}^\lambda_t$ be the portion of $\tilde{\mu}^\lambda - \lambda t \partial_z$ in $\Rnm \times (0,\infty)$. Then
\[
t \in [0,\infty) \mapsto \tilde{\mu}^\lambda_t
\]
is a spacelike Brakke flow with boundary $\tilde{\Gamma} := \Gamma \times (0,\infty)$ in $\Rnm \times (0,\infty)$. By the construction of $\Sigma^\lambda$, 
\begin{equation}\label{eq:AreaBoundTranslator}
	\sup_{\lambda > 0} \tilde{\mu}^\lambda(K) < \infty,
\end{equation}
\begin{equation}\label{eq:GradientBoundTranslator}
\inf_{\lambda > 0} \inf_{x \in K} S(\Tan(\tilde{\mu}^\lambda,x)) > 0
\end{equation}
and 
\begin{equation}\label{eq:ConormalBoundTranslator}
\inf_{\lambda > 0} \inf_{x \in K \cap \tilde{\Gamma}} \lnorm{\nu(\tilde{\mu}^\lambda,x)} > 0
\end{equation}
for every $K \Subset \Rnm \times (0,\infty)$. Therefore, by Theorem \ref{th:CompactnessSpacelikeBrakkeFlows} the spacelike Brakke flows $ (\tilde{\mu}^\lambda_t)_{t \geq 0}$ converge as $\lambda \to \infty$ (after passing to a subsequence) to a spacelike Brakke flow $(\tilde{\mu}_t)_{t \geq0}$ in $\Rnm \times (0,\infty)$ with boundary $\tilde{\Gamma}$. 

\begin{claim}
	The limit spacelike Brakke flow $t \in [0,\infty) \mapsto \tilde{\mu}_t$ in $\Rnm \times (0,\infty)$ with boundary $\tilde{\Gamma}$ is translation invariant.
\end{claim}

\begin{proof}
	Suppose that $\phi \in C_c^2(\Rnm \times (0,\infty);\R^+)$, and define 
	\[
		\phi^{\tau}(x,z) = \phi(x,z-\tau).
	\] 
	Let $t \in [0,\infty) \mapsto \tilde{\mu}^\lambda_t$ be the translating spacelike Brakke flow constructed above which converges to $\tilde{\mu}_t$ along some subsequence $\lambda \to \infty$. By construction, $\tilde{\mu}^\lambda_t(\phi^\tau) = \tilde{\mu}^\lambda_{t + \tau / \lambda}(\phi)$. There is a constant $c = c(\phi) > 0$, depending on $\phi$ but independent of $\lambda$, since the areas and gradients functions of $t \in [0,\infty) \mapsto \tilde{\mu}^\lambda_t$ are locally uniformly bounded as $\lambda \to \infty$, so that $t \in [0,\infty) \mapsto \tilde{\mu}^\lambda_t(\phi) + c(\phi) t$ is an increasing function of $t$. Hence if $t<s$, for $\lambda$ sufficiently large so that $t < t + \tau / \lambda < s$, we see that 
	\[ 
		\tilde{\mu}^\lambda_t(\phi) + c(\phi)t 
		\leq \tilde{\mu}^\lambda_t(\phi^\tau) + c(\phi) (t + \tau / \lambda)
		\leq 
		\tilde{\mu}^\lambda_s(\phi) + c(\phi)s. 
	\]
	Taking limits as $\lambda \to \infty$ along such subsequence, we have that
	\[ 
		\tilde{\mu}_t(\phi) \leq \tilde{\mu}_t(\phi^\tau) \leq \tilde{\mu}_s(\phi) + c(\phi)(s-t).
	\]
	Taking limits as $s \searrow t$, 
	\[ 
		\tilde{\mu}_t(\phi) \leq \tilde{\mu}_t(\phi^\tau) \leq \lim_{s \searrow t}\tilde{\mu}_s(\phi).
	\]
	Finally, for all but countably many $t$, the function $t \in [0,\infty) \mapsto \tilde{\mu}_t(\phi)$ is continuous, so $\tilde{\mu}_t(\phi^\tau) =  \tilde{\mu}_t(\phi)$ for a.e. $t \in [0,\infty)$. Therefore, 
	\[ 
		\tilde{\mu}_t = \mu_t \times (0,\infty) 
	\]
	except possibly for countably many $t$, where $\mu_t$ is a Radon measure in $\Rnm$. 
\end{proof}

Since $(\tilde{\mu}_t)_{t\geq 0}$ is a spacelike Brakke flow with boundary $\tilde{\Gamma}$ in $\Rnm \times (0,\infty)$, it follows that $(\mu_t)_{t \geq 0}$ is a spacelike Brakke flow with boundary $\Gamma$ in $\Rnm$.

Finally, we need to show that 
\[
\mu_t \rightharpoonup \sigma^n \resmes M_0
\]
as $t \to 0$. By Ilmanen's product Lemma \cite[Lemma 8.5]{Ilmanen94}, if $\theta \in C_c((0,\infty),\R^{+})$ satisfies $\int \theta = 1$, then $\mu_t(\phi(x)) = \tilde{\mu}_t(\phi(x)\theta(z))$. From \eqref{eq:LoweBoundArea}, we obtain
\begin{equation}\label{eq:LowerBoundAreaSections}
\tilde{\mu}_t(\Rnm \times (a,b)) \geq (b-a)\opn{area}(M_0).
\end{equation}
Approximating $\theta$ by step functions, 
\[
\mu_t(\Rnm) = \int d\mu_t = \int \theta(z)\, d\tilde{\mu}_t(x,z) \geq \opn{area}(M_0).
\]
that is, a spacelike Brakke flow constructed by elliptic regularization does not disappear immediately.

Let $b>0$, and $\pi : \Rnm \times \R \to \Rnm \times \{b\}$ be the orthogonal projection onto $\Rnm \times \{b\}$. We want to compute the volume element on $\pi(\Sigma^\lambda(0,b))$.
Let $(x_0,z_0) \in \Sigma^\lambda(0,b)$ such that $\opn{Tan}(\Sigma^\lambda,(x_0,z_0))$ exists. Since $\dim(\Tan(\Sigma^\lambda,(x_0,z_0)) \cap \Rnm) \in \{k,k+1\}$, we can take an orthonormal basis $\{v_1,\dots,v_{k+1}\}$ of $\Tan(\Sigma^\lambda,(x_0,z_0))$ such that $v_i = (u_i,0)$ for $i=1,\dots,k$ and $v_{k+1}=(u_{k+1},a)$ with $a^2 = \lnorm{\partial_z^\top}$. Thus, $\Tan(\pi(\Sigma^\lambda(0,b)),(x_0,b)) = \opn{span}\parent{\{u_1\dots,u_{k+1}\}}$. We observe that $\lnorm{\partial_z^\perp} = \lnorm{u_{k+1}}$ and, therefore,
\[
d\mu_{\pi(\Sigma^\lambda(0,b))} = \begin{cases}
     \sqrt{-\lnorm{\partial_z^\perp}} d\mu_{\Sigma^\lambda(0,b)}, & \text{if $u_{k+1} \neq 0$}, \\
     0, & \text{if $u_{k+1} = 0$},
\end{cases}
\]

Let $A^\lambda = \pi(\Sigma^\lambda(0,b))$. By the previous computation, 
\begin{align*}
	\opn{area}(A^\lambda) 
	&\leq  \int_{\Sigma^\lambda (0,b)} \sqrt{-\lnorm{\partial_z^\perp}} \, d\tilde{\mu}^\lambda\\ 
	&\leq \left(\int_{\Sigma^\lambda (0,b)}-\lnorm{\partial_z^\perp}\, d\tilde{\mu}^\lambda\right)^{1/2} \opn{area}(\Sigma^\lambda (0,b))^{1/2}.
\end{align*}
The area term is uniformly bounded and, by \eqref{eq:IlmanenMonotonicity}, 
\[
	\int_{\Sigma^\lambda(0,b)} -\lnorm{\partial_z^\perp}\,d\mu^\varepsilon 
	\leq \lambda^{-1}  \int_{\Sigma^\lambda_b} \sqrt{\lnorm{\partial_z^\top}}\, d\mu_{\Sigma^\lambda_b}.
\]
Thus, by \eqref{eq:GradientBoundTranslator}, we see that $\opn{area}(A^\lambda) \to 0$ as $\lambda \to \infty$. We can use this to see that the area of the red region in Figure \ref{Fig:ConvergenceInitialData} is tending to zero. Therefore, by Ascoli-Arzelà theorem, there exist a further subsequence such that $\Sigma^\lambda(0,b)$ converge to $M_0  \times (0,b)$ strongly in $C^0(\Omega \times (0,b))$ and weakly in $H^1(\Omega \times (0,b))$ on compact subsets. Moreover, by \eqref{eq:LowerBoundAreaSections} and the upper semicontinuity of the area functional,  
\begin{align*}
	\opn{area}(M_0  \times(0,b)) &\geq \limsup_{\lambda \to \infty} \opn{area}(\Sigma^\lambda(0,b))\\
	&\geq \liminf_{\lambda \to \infty} \opn{area}(\Sigma^\lambda(0,b)) \\
	&=b\opn{area}(M_0)\\ 
	&= \opn{area}(M_0  \times(0,b)).
\end{align*}
However, since $\Sigma^\lambda \in \mcal{G}$, we have that $\Sigma^\lambda(0,b)$ remains in a compact subset for each $\lambda > 0$. By Prokhorov's theorem, if the total area converges, then the measures must converge. This combines to show that  
\[
\tilde{\mu}_t \rightharpoonup (\sigma^n \resmes M_0) \times (0,\infty) 
\]
as $t \to 0$. 
\end{proof}

\begin{figure}[htb]\label{Fig:ConvergenceInitialData}
	\centering
	\includegraphics[width=0.8\textwidth]{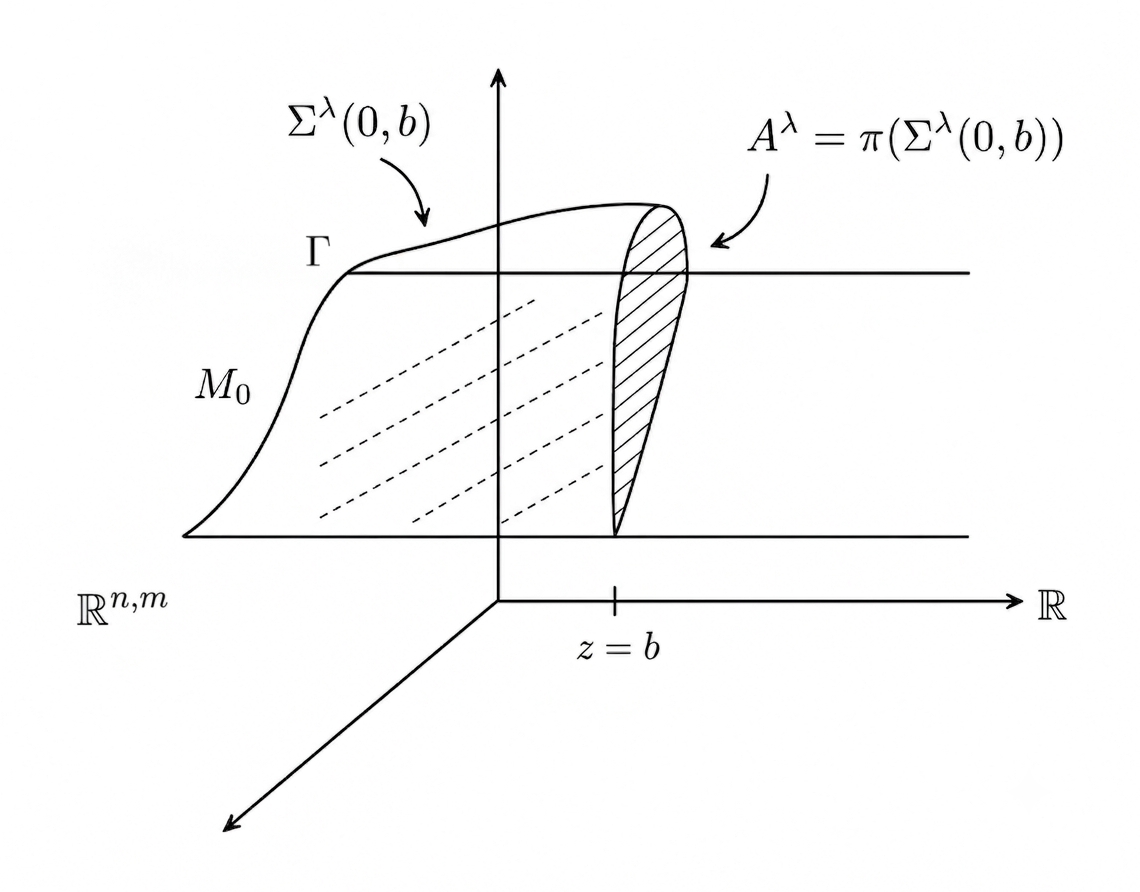}
	\caption{Convergence to the initial condition.}
\end{figure}

\section{Monotonicity with boundary}

Let $\Gamma$ be a smooth, properly embedded, $(n-1)$-dimensional, spacelike submanifold in $\Rnm$. For $x_0 \in \Rnm$ the interior cone over $\Gamma$ with vertex $x_0$ is 
\[ 
I_{\Gamma,x_0} = \{ x_0 + s(x-x_0) : x \in \Gamma,\, 0 \leq s \leq 1\}. 
\]
The multiplicity $\theta(p)=\theta_{\Gamma,p}(p)$ of the interior cone at a point $p \in \Rnm$ is the number of points $(x,s) \in \Gamma \times [0,1]$ such that $x_0 + s(x-x_0) = p$. If the interior cone is spacelike, counting multiplicity, determines a Radon measure $I = I_{\Gamma, x_0}$ on $\Rnm$ given by 
\[
dI_{\Gamma, x_0} = \theta_{\Gamma, x_0}d\sigma^n.
\] 

Consider the function defined by 
\[
	\psi(x,t) = \frac{1}{(-4\pi t)^{n/2}}e^{\frac{\lnorm{x}}{4t}} 
\]
for $x \in \Rnm$ and $t <0$, and its translates
\[
	\psi_{(x_0,t_0)}(x,t) = \psi(x-x_0,t-t_0) = \frac{1}{(4\pi(t_0-t))^{n/2}} e^{\frac{-\lnorm{x-x_0}}{4(t_0-t)}}. 
\]
for $x_0 \in \Rnm$ and $t < t_0$. 

\begin{theorem}[Huisken's Monotonicity]
    \label{th:Monotonicity}
	Let $(\mu_t)_{t \in [a,b]}$ be an $n$-dimensional spacelike Brakke flow with boundary $\Gamma$ in $U \Subset \Rnm$. Let $x_0 \in \Rnm$ and $t_0 > b$ such that $I_{\Gamma,x_0}$ is a spacelike cone. Then 
	\[
	(\mu_t + I_{\Gamma,x_0})\psi_{(x_0,t_0)}(\cdot,t)
	\]
	is an increasing function of $t \in [a,b]$. Indeed,
	\begin{gather*} 
		(\mu_b + I_{\Gamma,x_0})\psi_{(x_0,t_0)}(\cdot,b) - (\mu_a + I_{\Gamma,x_0})\psi_{(x_0,t_0)}(\cdot,a) \\
		\geq \int_a^b\int -\lnorm{\left( H-\frac{\nabla\psi_{(x_0,t_0)}}{\psi_{(x_0,t_0)}}\right)^\perp}\psi_{(x_0,t_0)} \,d\mu_tdt \\
		+ \int_a^b\int \abs{\esc{\nu_M + \nu_I}{\nabla\psi_{(x_0,t_0)}}}\,d\Gamma dt.   
	\end{gather*}
	Furthermore, if $\mu_b\psi_{(x_0,t_0)}(\cdot,b) = \mu_a\psi_{(x_0,t_0)}(\cdot,a)$, then for almost every $t \in [a,b]$,
	\[ H = \frac{\nabla^\perp\psi_{(x_0,t_0)}}{\psi_{(x_0,t_0)}} \]
	holds $\mu_t$-almost everywhere, and
	\[ \nu_{\mu_t} = \frac{x-x_0}{(x-x_0)^\perp} \]
	holds for almost every $x \in \Gamma$ for which $(x-x_0)^\perp$ is nonzero, where $(x-x_0)^\perp$ is the projection of $(x-x_0)$ to $\Tan(\Gamma,x)^\perp$.  
\end{theorem} 

\begin{proof}
	By translating, it suffices to consider the case when $(x_0,t_0) = (0,0)$ and $a \leq b < 0$. A direct computation shows that
	\begin{align*}
		\nabla \psi &= \frac{\psi}{2t}x\\
		\Div_{\mu_t}\nabla\psi & = \frac{\psi n}{2t} + \frac{\psi}{4t^2}\esc{x^\top}{x^\top}\\
		\frac{\partial\psi}{\partial t} &= -\frac{\psi n}{2t} - \frac{\psi}{4t^2}\esc{x}{x}\\
		\frac{\esc{\nabla^\perp\psi}{\nabla^\perp\psi}}{\psi} &= \frac{\psi}{4t^2}\esc{x^\perp}{x^\perp}.
	\end{align*}
	Thus 
	\[ 
	Q_{\mu_t}(\psi) = \frac{\partial\psi}{\partial t} + \Div_{\mu_t}\nabla\psi + \frac{\esc{\nabla^\perp\psi}{\nabla^\perp\psi}}{\psi} = 0. 
	\]
	By the definition of spacelike Brakke flow,
	\begin{gather*}
		\mu_b(\psi) - \mu_a(\psi)\\
		\geq \int_a^b\int -\psi \esc{H}{H} + \esc{H}{\nabla\psi} + \frac{\partial\psi}{\partial t}\, d\mu_t dt\\ 
		= \int_a^b\int -\psi \esc{H}{H} + 2\esc{H}{\nabla\psi} +\Div_{\mu_t}\nabla\psi + \frac{\partial\psi}{\partial t}\, d\mu_t dt
		- \int_a^b\int \esc{\nu_M}{\nabla\psi}\,d\Gamma dt\\
		= \int_a^b\int -\psi \lnorm{\left(H-\frac{\nabla\psi}{\psi}\right)^\perp} + Q_{\mu_t}(\psi) \, d\mu_t dt
		- \int_a^b\int \esc{\nu_M}{\nabla\psi}\,d\Gamma dt\\
	\end{gather*}
	
	For notational simplicity, let us assume the interior cone over $\Gamma$ (with vertex $0$) is embedded, i.e., that the multiplicity $\theta = \theta_{\Gamma,0}$ is $1$ at all points on the cone. We will use $I=I_{\Gamma,0}$ to denote both the interior cone and the associated Radon measure.
	
	The previous computation shows that on $I$, 
	\[ \frac{\partial \psi}{\partial t} = -\Div_{I}\nabla\psi. \]
	Thus
	\begin{align*}
		\frac{d}{dt} I\psi(\cdot,t) &= \frac{d}{dt} \int_I \psi\,d\sigma^n \\
		&= \int_I \frac{\partial\psi}{\partial t}\psi\,d\sigma^n\\
		&= - \int_I \Div_I \nabla\psi\,\sigma^n\\
		&= \int_I \esc{H_I}{\nabla\psi}\, d\sigma^n - \int_\Gamma \esc{\nu_I}{\nabla\psi}\,d\Gamma\\
		&= - \int_\Gamma \esc{\nu_I}{\nabla\psi}\,d\Gamma,
	\end{align*}
	where we have used that $\esc{H_I}{\nabla\psi} = 0$, since $H_I$ is perpendicular to $I$ by definition and $\nabla\psi(x,t) / \psi(x,t)= - x / 2t$ is tangent to $I$, because $I$ is a portion of a cone with vertex at the origin. Thus
	\[  
	I\psi(\cdot,b) -  I\psi(\cdot,a) = \int_a^b\int \frac{d}{dt} I\psi(\cdot,t)\,dt = -\int_a^b\int \esc{\nu_I}{\nabla\psi}\,d\Gamma dt. 
	\] 
	
	Now, if $x \in \Tan(\Gamma,x)$, then 
	\[ \esc{u + \nu_I}{\nabla\psi} = 0  \]
	for all $u \in \Tan(\Gamma,x)^\perp$. If $x^\perp \neq 0$, taking into account that $\Tan(\Gamma,x)^\perp$ is a one dimensional vector space and that $\nu_I$ points outside the origin, 
	\[ \nu_I = \frac{x^\perp}{\sqrt{\esc{x^\perp}{x^\perp}}}, \]
	where $x^\perp$ is the projection of $x$ to $\Tan(\Gamma,x)^\perp$. Finally, if $u \in \Tan(\Gamma,x)^\perp$ with $\esc{u}{u} \leq 1$, then 
	\[ \esc{u + \nu_I}{\nabla\psi} \leq 0 \]
	with equality if and only if $u = -\nu_I$, due to the fact that 
	\[
	\esc{u + \nu_I}{\nabla\psi} \leq 0 \iff \esc{u+\nu_I}{-x^\perp} \leq 0 \iff -(\esc{u}{\nu_I} + 1) \leq 0
	\]
	and the Cauchy-Schwarz inequality. Adding the previous inequalities,
	\begin{gather*}
		(\mu_b+I)\psi(\cdot,b) - (\mu_a+I)\psi(\cdot,a)\\
		\geq \int_a^b\int -\lnorm{\left(H-\frac{\nabla\psi}{\psi}\right)^\perp}\psi \,d\mu_t dt 
		+ \int_a^b\int \abs{\esc{\nu_M + \nu_I}{\nabla\psi}}\,d\Gamma dt. \qedhere 
	\end{gather*} 
\end{proof}

In the noncompact case, we need additional assumptions on the  causal character of $\mu_t$ and the integrability of the kernel $\psi$. To insert $\psi$ into the above formula, for any cutoff function $\phi \in C_c^\infty(\Rnm;\R^+)$, we calculate
\[ 
	Q_{\mu_t}(\phi\psi) = \phi Q_{\mu_t}(\psi) + \psi Q_{\mu_t}(\phi) + 2\esc{\nabla^\top\phi}{\nabla^\top \psi} = \psi Q_{\mu_t}(\phi) + 2\esc{\nabla^\top\phi}{\nabla^\top \psi}
\]
and 
\[ 
	\esc{\nu_M + \nu_I}{\nabla(\phi\psi)} = \psi\esc{\nu_M + \nu_I}{\nabla\phi} + \phi \esc{\nu_M + \nu_I}{\nabla \psi}\]
We now let $\phi = \chi_R$, where 
\[ \chi_{B_R} \leq \chi_R \leq \chi_{B_{2R}} \]
($\chi_{B_R}$ is the characteristic function for the ball $B_R$) and
\[ R \enorm{\nabla\chi_R} + R^2\enorm{\nabla^2\chi_R} \leq c_0, \]
that is, we choose 
\[ 
	\chi_R = \eta\left( \frac{x_1^2 + \cdots + x_{n+m}^2}{R^2} \right) 
\]
with $\eta : \R \to [0,1]$ a smooth decreasing function such that $\eta(x) = 1$ for $x \leq 1$ and $\eta(x) = 0$ for $x \geq 4$. Recalling the estimate $\eqref{eq:GradientEstimateI}$, if $v^2(\mu_t,\cdot) \leq c_{v}$ for all $t \in [a,b]$,
\begin{align*}
	\abs{Q_{\mu_t}(\psi\chi_R)} &\leq \left( \frac{n(n + v^2)c_0}{R^2} + \frac{(\eta')^2}{\eta}\frac{4mv^2}{R^2} \right)\psi \chi_{B_{2R} \setminus B_R} + \frac{2nv^2\abs{\eta'}}{-t}\psi \chi_{B_{2R} \setminus B_R}\\	
	&\leq c(n,m,c_v,c_0)\left(\frac{1}{R^2} + \frac{1}{-b} \right)\psi \chi_{B_{2R} \setminus B_R}
\end{align*}  
and, if $\Gamma$ is uniformly spacelike, 
\begin{align*}
		\abs{\esc{\nu_M + \nu_I}{\nabla \phi}}\psi &\leq  \abs{\esc{\nu_M}{\nabla\phi}} \psi  + \abs{\esc{\nu_I}{\nabla\phi}} \psi \\
		&\leq \frac{c_2(\Gamma,c_0)}{R}\psi \chi_{B_{2R} \setminus B_R} 
\end{align*}
Inserting $\chi_R\psi$ in the spacelike Brakke flow inequality, we obtain
\begin{gather*}
	\int \chi_R\psi \,d\mu_b - \int \chi_R\psi \,d\mu_a\\
	\geq \int_a^b\int -\chi_R\psi \lnorm{\left(H-\frac{\nabla\psi}{\psi} - \frac{\nabla\chi_R}{\chi_R}\right)^\perp}\, d\mu_tdt + 
	\int_a^b \int Q_{\mu_t}(\chi_R\psi)\, d\mu_tdt\\
	+ \int_a^b\int \abs{\esc{\nu_M+\nu_I}{\nabla(\chi_R\psi)}}\, d\Gamma dt.
\end{gather*}
Finally, if
\[
\sup_{t \in [a,b]}\int \psi \,d\mu_t < \infty,
\]
the result follows by the monotone and dominated convergence theorem.

\begin{definition}\label{def:UniformlySpacelikeBrakkeFlow}
	An $n$-dimensional spacelike Brakke flow $(\mu_t)_{t \in I}$ with boundary $\Gamma$ in $\Rnm$ is called uniformly spacelike
	if it has the properties described in the monotonicity formula:
	\begin{enumerate}[(1)]
		\item $\Gamma$ is uniformly spacelike
		\item For almost every $t \in I$, the interior cone $I_{x_0,\Gamma}$ is spacelike for $\mu_t$ almost every $x_0 \in \Rnm$.
		\item For any $[a,b] \subset I$, there exists a positive constant $c_v=c_v(a,b)$ such that
		\[ 
			\sup_{t \in [a,b]}\sup_{x \in \spt\mu_t} v^2(\mu_t,x) \leq c_v.
		\]
		\item For any $[a,b] \subset I$, $x_0 \in \Rnm$ and $t_0>b$, 
		\[
			\lim_{R \to \infty} \sup_{t \in [a,b]}\int_{\Rnm \setminus B_R} \psi_{(x_0,t_0)}\, d\mu_t = 0.
		\]
	\end{enumerate} 
\end{definition}

\section{Tangent flows}

Consider an $n$-dimensional spacelike Brakke flow $(\mu_t)_{t \geq 0}$ in $\Rnm$ with boundary $\Gamma$ such that 
\begin{equation}\label{eq:TangentFlowsMass}
	\sup_{t \in [a,b]}\mu_t(K) < \infty, 
\end{equation}
\[
\sup_{t \in [a,b]} \sup \{v^2(\mu_t, x) : x \in K\} < \infty
\] 
and
\[
\inf_{t \in [a,b]} \inf \{\lnorm{\nu(x,t)} : x \in K \cap \Gamma, \, \nu(x,t) \neq 0 \}  > 0
\]
for every $K \Subset \Rnm$. We now discuss the weak existence of tangent flows at a spacetime point $X_0 = (x_0,t_0) \in \Rnm \times (0,+\infty)$. 

For $\lambda > 0$, let $\mu^\lambda_t$ be the parabolically dilated measures at a point $X_0$ given by  
\begin{equation}\label{eq:ParabolicDilation}
	\mu^{X_0,\lambda}_t (A) = \lambda^n\mu_{t_0 + \lambda^{-2}t}(\lambda^{-1} A + x_0),
\end{equation}
for $t \in \lambda^2[-t_0,+\infty)$. Then, 
\[
t \in \lambda^2[-t_0,\infty) \mapsto \mu^{X_0,\lambda}_t
\] 
is again a $n$-dimensional spacelike Brakke flow in $\Rnm$ with boundary $\lambda \Gamma$.  

\begin{proposition}
    \label{prop:TangentFlows}
	For every sequence $\lambda_i \to \infty$, there is a subsequence $\lambda_{i(j)}$ and a limiting flow $(\mu'_t)_{t \in \R}$ such that $\mu^{X_0,\lambda_{i(j)}}_t \rightharpoonup \mu'_t$ in the sense of Radon measures for all $t \in \R$. If $x_0 \not\in \Gamma$, it is a spacelike Brakke flow in $\Rnm$. If $x_0 \in \Gamma$, it is a spacelike Brakke flow with boundary $\Tan(\Gamma,x_0)$ in $\Rnm$.  
\end{proposition}

\begin{proof}
	By $\eqref{eq:TangentFlowsMass}$ and $\eqref{eq:ParabolicDilation}$,
	\[
	\sup_{t \in [a,b]}\mu^{X_0,\lambda_i}_t(K) < \infty
	\]
	for every $K \Subset \Rnm$. Moreover, since the parabolic dilations preserve angles,
	\[
	v^2(\mu^{X_0,\lambda_i}_t, \lambda_i (x-x_0)) = v^2(\mu_t, x)  
	\]
	for every $x \in \spt \mu_t$ such that $\Tan(\mu_t,x)$ exists, and 
	\[ 
	\nu(\mu^{X_0,\lambda_i}_t,\lambda_i (x-x_0)) = \nu(\mu_t, x) 
	\]
	for $\sigma^{n-1}$ almost every $x \in \Gamma$. Then by the compactness theorem, there exists a subsequence $\lambda_{i(j)} \to \infty$ and a limit spacelike Brakke flow $(\mu'_t)_{t \in \R}$ with boundary 
	\[
	\Gamma' = \begin{cases}
		\emptyset, & \text{if $x_0 \not\in \Gamma$}\\
		\Tan(\Gamma,0), & \text{if $x_0 \in \Gamma$.}
	\end{cases}
	\] 
\end{proof}

\section{Gaussian density}

Suppose that $\Gamma$ is a smooth, properly embedded $(n-1)$-dimensional spacelike submanifold in $\Rnm$, and suppose that $M$ is any Radon measure on $\Rnm$. If $a \in \Rnm$ such that $I_{\Gamma,a}$ is spacelike, we let $[M,\Gamma, a]$ be the Radon measure
\[ M + I_{\Gamma,a}, \]
and define 
\[ 
	\Theta(M,\Gamma,a,r) = \int \frac{1}{(4\pi r^2)^{n/2}}e^{\frac{-\lnorm{x-a}}{4r^2}}\,d[M,\Gamma, a](x) + \frac{1}{2}\chi_{\Gamma}(a), 
\]
where 
\[ 
\chi_\Gamma(x) =
\begin{cases}
	1 & \text{if } x \in \Gamma,\\
	0 & \text{if } x \not\in \Gamma.
\end{cases}
 \]

As in the Euclidean case \cite{White21}, the term $\frac{1}{2}\chi_\Gamma(x)$ is necessary to make $\Theta_{\textrm{gauss}}(M,\Gamma,a,r)$ continuous as a function of $a$.

Let $\mcal{M} = (\mu_t)_{t \geq 0}$ be a uniform spacelike Brakke flow with boundary $\Gamma$ in $\Rnm$ and $x_0 \in \Rnm$, $t_0>0$ and $r>0$. The monotonicity formula implies that
\[ 
\Theta(\mcal{M},\Gamma,(x_0,t_0),r) =	\Theta(\mu_{t_0-r^2},\Gamma,x_0,r) 
\]
is decreasing in $r \in [0,\sqrt{t_0})$ as long as $I_{\Gamma,x_0}$ is spacelike. Hence, as $r \to 0$, the limit exists, so we can define the Gaussian density of $\mcal{M}$ at $(x_0,t_0)$ as 
\[
\Theta(\mcal{M},\Gamma,(x_0,t_0)) = \lim_{r\to 0} \Theta(\mu_{t_0-r^2},\Gamma,x_0,r) \in [0,+\infty].
\]	

Now suppose that $\mcal{M} = (\mu_t)_{t \in (-\infty,T]}$ is a uniform spacelike Brakke flow with boundary $\Gamma$ in $\Rnm$. Then we may set
\[
\Theta(\mcal{M},\Gamma, \infty) = \lim_{r \to \infty}\Theta(\mcal{M},\Gamma, (x_0,t_0),r),
\]
provided the limit, that exists and it is finite by the monotonicity formula, is independent of $(x_0,t_0)$. 

\begin{proposition}\label{Prop:LowerSemicontinuityI}
	Let $(\mcal{M}_i)_{i \in \mathbf N}$ be a sequence of uniformly spacelike Brakke flows with boundary $\Gamma_i$ in $\Rnm$, independently of $i \in \mathbf{N}$. Suppose that
	\[
		\inf_i \Theta(\mcal{M}_i,\Gamma_i, X_i,r_i) \geq \Theta_0
	\]
	for every $X_i \to 0$, $r_i \to 0$ subject to $r_i \leq R_i$ for some sequence $R_i$. Let $\lambda_i \to \infty$ be another sequence, with $\lambda_i R_i \to \infty$, and suppose the dilated flows $\mcal{D}_{\lambda_i}\mcal{M}_i$ converges to a limit spacelike Brakke flow $\mcal{M}$ in $\Rnm$ with boundary $\Gamma$. Then we have 
	\[
	\Theta(\mcal{M},\Gamma, X,r) \geq \Theta_0
	\]
	for all $X \in \mcal{M}$ and $r>0$. 
\end{proposition}

\begin{proof}
	Let $X \in \mcal{M}$ and $r>0$. By making a spacetime translation, we can assume that $X=(0,0)$. By property (4) in Definition \ref{def:UniformlySpacelikeBrakkeFlow}, the sequence of measures $\psi \mcal{D}_{\lambda_i}\mcal{M}_i(-r^2)$ is tight for each fixed $r>0$. Therefore  
	\begin{align*}
		\int_{\mcal{M}(-r^2)} (4\pi r^2)^{-n/2} e^{\frac{-\lnorm{x}}{4r^2}} + \frac{1}{2}\chi_{\Gamma}(0) 
		&= \lim_{i} \left( \int_{\mcal{D}_{\lambda_i}\mcal{M}_i(-r^2)}(4\pi r^2)^{-n/2} e^{\frac{-\lnorm{x}}{4r^2}} + \frac{1}{2}\chi_{\mcal{D}_{\lambda_i}\Gamma_i}(0) \right) \\
		&=\lim_{i} \Theta(\mcal{D}_{\lambda_i}\mcal{M}_i, \mcal{D}_{\lambda_i}\Gamma_i, (0,0),r)\\
		&=\lim_{i} \Theta(\mcal{M}_i, \Gamma_i, (0,0),r/\lambda_i)\\
        &\geq \Theta_0,
	\end{align*} 
    where in the last inequality we have used that $\lambda_i R_i \to \infty$. 
\end{proof}

\begin{proposition}\label{Prop:LowerSemicontinuityII}
	Let $(\mcal{M}_i)_{i \in \mathbf N}$ be a sequence of uniformly spacelike Brakke flows with boundary $\Gamma_i$ in $\Rnm$, independently of $i \in \mathbf{N}$. Suppose that $\Gamma_i$ converges to $\Gamma$ in $C^\infty$, and that $\mcal{M}_i$ converges to a spacelike Brakke flow $\mcal{M}$ with boundary $\Gamma$ in $\Rnm$. Then for every $X_i \to X$ and $r_i \to 0$ we have 
	\[
		\liminf_i \Theta(\mcal{M}_i,\Gamma_i, X_i,r_i) \geq \Theta(\mcal{M},\Gamma, X).
	\]
\end{proposition}

\begin{proof}
	By making a spacetime translation, we can assume that $X_i=X=(0,0)$. Then, for $r>0$ and $i$ sufficiently large, we have that $r_i < r$. Thus 
	\[
	\Theta(\mcal{M},\Gamma,(0,0),r) = \liminf_i \Theta(\mcal{M}_i,\Gamma_i,(0,0),r) \leq \liminf_i \Theta(\mcal{M}_i,\Gamma_i,(0,0),r_i).
	\] 
	for any $r>0$. Letting $r \to 0$, the proposition follows.
\end{proof}

\section{Local regularity}

Finally, we adapt White's local regularity theorem to the pseudo-Euclidean setting. As in \cite{White05}, let $\RnmR$ be the spacetime and, if $X = (x,t)$ is a point in spacetime, denote by $\abs{X}$ its parabolic norm
\[ 
\abs{X} = \abs{(x,t)} = \max\{ \abs{x}_{\R^{n+m}}, \abs{t}^{1/2} \}.
\]
For $\lambda>0$, let $\mcal{D}_\lambda : \RnmR \to \RnmR$ denote the parabolic dilation 
\[
\mcal{D}_\lambda(x,t) = (\lambda x, \lambda^2 t).
\]
The graph of a function $u \colon B_1^n(0) \times (-1,1) \rightarrow \R^m$ is the set 
\[ 
\opn{graph}(u) = \{(x,u(x,t),t) : (x,t) \in B_1^n(0) \times (-1,1)\}.
\]
Let $\mcal{M}$ be a $(n+1)$-dimensional smooth submanifold (with possible boundary) of $\RnmR$ in the classical sense. We define the parabolic $C^{2,\alpha}$-regularity scale of $\mcal{M}$ at $X$, denoted by $K_{2,\alpha}(\mcal{M},X)$ to be the infimum $\lambda > 0$ so that after a suitable rotation in space, the dilated and translated submanifold 
\[ 
	\mcal{D}_\lambda(\mcal{M}-X) \cap (B_1^n(0) \times B_1^m(0) \times (-1,1)) = \opn{graph}(u) 
\]  
for some $u : B_1^n(0) \times (-1,1) \to \R^m$ whose parabolic $C^{2,\alpha}$ norm is $\leq 1$. The $C^{2,\alpha}$-regularity scale $K_{2,\alpha}(\mcal{M}, \cdot)$ verifies 
\[
K_{2,\alpha}(\mcal{D}_\lambda\mcal{M},\mcal{D}_\lambda X)= \lambda
^{-1} K_{2,\alpha}(\mcal{M}, X),
\]
but we will also need a scale invariant version of $K_{2,\alpha}$. If 
\[ d(X,U) = \inf\{ \abs{X-Y} : Y \in U^c \}, \]
then 
\[
d(X,U)K_{2,\alpha}(\mcal{M},X)
\]
is scale invariant.

\begin{definition}
	Let $\mcal{M}$ be a smooth, spacelike mean curvature flow in an open subset $U$ of $\RnmR$. Then 
	\[
	K_{2,\alpha; U}(\mcal{M}) = \sup_{X \in \mcal{M}} d(X,U) \cdot K_{2,\alpha}(\mcal{M},X). 
	\]
	Of course, $K_{2,\alpha; U}(\mcal{M})$ is scale invariant. 
\end{definition}

\begin{proposition}
	Let $\mathcal{M}$ be a smooth, uniformly spacelike mean curvature flow (without boundary) in $\RnmR$. Then $\Theta(\mathcal{M},\infty) \leq 1$, with equality if, and only if,  
	\begin{equation}\label{eq:SpacelikePlaneFlow}
		\mcal{M} = P \times (-\infty,T] 
	\end{equation}
	for some $n$-dimensional spacelike affine plane $P$ of $\Rnm$ and some $T \in (-\infty,\infty]$. 
\end{proposition}

\begin{proof}
	Let $X \in \mcal{M}$. Then, the dilates $D_\lambda(\mathcal{M}-X)$ converges, as $\lambda \to \infty$, to a limit flow $\mathcal{M}'$ of the form $\eqref{eq:SpacelikePlaneFlow}$. Hence, for $r>1$, 
	\begin{align*}
		\Theta(\mathcal{M}',(0,0),1) &\geq \lim_{\lambda \to \infty} \Theta(\mathcal{D}_\lambda(\mathcal{M}-X),(0,0),r)\\
		&= \lim_{\lambda \to \infty} \Theta({M}-X,(0,0),r/\lambda)\\
		&= \lim_{\lambda \to \infty} \Theta({M},X,r/\lambda)\\
		&= \Theta(\mathcal{M},X).
	\end{align*}   
	By applying a suitable isometry $\Phi : \Rnm \rightarrow \Rnm$ such that $\Phi(P) = \R^n \times \{0\}^m$,
	\begin{align*}
		\Theta(\mathcal{M}',(0,0),1) = \int_{\R^n \times \{0\}^m} \frac{1}{(4\pi)^{n/2}} \exp \left(\frac{-x_1^2 - \cdots - x_n^2}{4}\right)\, d\mu_{\R^n \times \{0\}^m} = 1.
	\end{align*}
	Therefore, $\Theta(\mathcal{M},X) \leq 1$ and, by monotonicity, $\Theta(\mathcal{M},\infty) \leq 1$. 
	
	Furthermore, if $\Theta(\mathcal{M},\infty) = 1$, then $\Theta(\mathcal{M},\infty) = \Theta(\mathcal{M},X) = 1$ for every $X=(x_0,t_0) \in \mathcal{M}$. The proof of the monotonicity formula shows that the flow 
	\[
		\mathcal{M}' = (\mathcal{M} - X) \cap \{(x,t) \in \Rnm \times \R : t \leq 0\}
	\]
	is invariant under parabolic dilations, that is, $\mathcal{M}' = \mathcal{D}_\lambda\mathcal{M}'$. Taking the limit of this equation as $\lambda \to \infty$ shows that $\mathcal{M}'$ has the form $\eqref{eq:SpacelikePlaneFlow}$, which implies that $\{ (x,t) \in \mathcal{M} : t \leq t_0 \}$ has the form $\eqref{eq:SpacelikePlaneFlow}$. Since this is true for every $X \in \mathcal{M}$, then all of $\mathcal{M}$ has the form \eqref{eq:SpacelikePlaneFlow}.  
\end{proof}

In the following theorem, we denote by $P(X_0,r)$ the parabolic cylinder $P(X_0,r) = B_r(x_0) \times (t_0 - r^2, t_0 + r^2)$, where $X_0 = (x_0,t_0)$ and $r \in (0,\infty]$. If $r = \infty$, then $P(X_0,\infty)$ should be interpreted as $\RnmR$.   

\begin{theorem}[White's local regularity theorem]\label{th:LocalRegularity}
	For $0<\alpha<1$, there exist positive numbers $\varepsilon=\varepsilon(n,m,\alpha)$, $\delta = \delta(n,m,\alpha)$ and $C=C(n,m,\alpha)$ with the following property. Suppose $\mcal{M}$ is a graphical, smooth, spacelike mean curvature flow in $P(X_0,R)$ such that
	\[
	\sup_{X \in \mcal{M} \cap P(X_0,r)} \Theta(\mcal{M},X,r) \geq 1 - \ve \\
	\]
	and 
	\[
	\sup_{X \in \mcal{M} \cap P(X_0,r)} v^2(\mcal{M},X) \leq m + \delta
	\]
	for some $r \in (0,R)$. Then 
	\[
	K_{2,\alpha;P(X_0,r/2)} (\mcal{M}) \leq C.
	\]
\end{theorem}

\begin{proof}
	Let $\ol{\varepsilon}$ and $\ol{\delta}$ the infimum of numbers $\varepsilon > 0$ and $\delta > 0$ respectively for which the theorem fails, i.e., for which there is no appropriate $C < \infty$. We must show that $\ol{\varepsilon}, \ol{\delta} > 0$. Let $\varepsilon_i > \ol{\varepsilon}$ and $\delta_i > \ol{\delta}$ be sequences of numbers converging to $\ol{\varepsilon}$ and $\ol{\delta}$ respectively. Then there are sequences $\mcal{M}_i$ and $P(X_i,R_i)$ such that $\mcal{M}_i$ is a smooth, graphical, spacelike mean curvature flow in $P(X_i,R_i)$ and such that 
	\begin{align*}
		\sup_{X \in \mcal{M}_i \cap P(X_i,r_i)} \Theta(\mcal{M}_i,X, r_i) &\geq 1 - \varepsilon_i, \\
		\sup_{X \in \mcal{M}_i \cap P(X_i,r_i)} v^2(\mcal{M}_i, X) &\leq m + \delta_i,  
	\end{align*}
	for some $r_i \in (0,R_i)$ and 
	\[ K_{2,\alpha;P(X_i,r_i/2)}(\mcal{M}_i) \to \infty \]
	Let $U_i = P(X_i,r_i/2)$. By \cite[Proposition 2.8]{White05}, we can assume that 
	\[
	K_{2,\alpha;U_i}(\mcal{M}_i) = s_i < \infty
	\]
	for each $i$. Choose $X_i \in \mcal{M}_i$ so that 
	\[
	d(X_i,U_i)K_{2,\alpha}(\mcal{M}_i,X_i) > \frac{1}{2}s_i.
	\]
	By translating and parabolic dilating, we may assume that $X_i = (0,0)$ and $K_{2,\alpha}(\mcal{M}_i,(0,0)) = 1$, which implies $d((0,0),U_i) \to \infty$. Now let $X \in \mcal{M}_i$. Then 
	\[
	d(X,U_i)K_{2,\alpha}(\mcal{M}_i,X) \leq s_i \leq 2d((0,0),U_i)K_{2,\alpha}(\mcal{M}_i,(0,0)) = 2d((0,0),U_i).
	\]
	Thus 
	\[
	K_{2,\alpha}(\mcal{M}_i,X) \leq 2 \frac{d((0,0),U_i)}{d(X,U_i)} \leq 2 \frac{d((0,0),U_i)}{d((0,0),U_i)-\abs{X}} = 2 \left( 1-\frac{\abs{X}}{d((0,0),U_i)} \right)^{-1}
	\]
	provided the right-hand is positive. Since $d((0,0),U_i) \to \infty$, we have that $K_{2,\alpha}(\mcal{M}_i,\cdot)$ is uniformly bounded as $i \to \infty$ on compact subsets of spacetime. Thus by Ascoli-Arzelà theorem (c.f \cite[Theorem 2.7]{White05}), there exist a subsequence of the $\mcal{M}_i$ (which we denote by the same index) that converges locally to a limit spacelike mean curvature flow $\mcal{M}$ in $\Rnm \times \R$, due to the uniform gradient bound and the fact that  $d((0,0),U_i) \to \infty$. By Proposition \eqref{Prop:LowerSemicontinuityI},
	\begin{align*}
		&\Theta(\mcal{M},X,r) \geq 1 - \ol{\varepsilon}, \\
		&v^2(\mathcal{M},X) \leq m + \ol{\delta},
	\end{align*}
	for all $X \in \mcal{M}$ and $r>0$. 
	
	Now suppose that $\ol{\varepsilon} = 0$ or $\ol{\delta}=0$. We will show that this leads a contradiction. If $\ol{\delta} = 0$, then $v^2(\mcal{M},X) = m$ for all $X \in \mcal{M}$, which implies that $\mathcal{M}$ has the form
	\[
	\mcal{M} = \R^n \times \{0\}^m \times (-\infty,T],
	\]
	for some $T \in [0,\infty]$. Now, if $\ol{\varepsilon} = 0$, the monotonicity formula implies that (after a suitable rotation) $\mathcal{M}$ has also the previous form. By the $C^2$ convergence, there exists $\rho_i \to \infty$ such that
	\[ 
	\mcal{M}_i \cap (B^n_{\rho_i}(0) \times B^m_{\rho_i}(0) \times (-\rho_i,\rho_i))
	\]
	is the graph of a function
	\[
	u_i \colon B^n_{\rho_i}(0) \times ((-\rho_i,\rho_i) \cap (-\infty,T_i]) \rightarrow \R^m
	\]
	for some $T_i \geq 0$ converging to $T$, and with $\norm{u_i}_{C^2} \to 0$. Of course, the $u_i$ satisfy the nonparametric mean curvature flow equation,
	\[
	\frac{\partial u_i}{\partial t} - \Delta u_i = f_i, 
	\] 
	where
	\[
	f_i = \sum_{1 \leq j,k \leq n} \frac{D_ju_iD_ku_i}{1-\abs{Du_i}_{\R^n}^2}D_{jk}u_i.
	\]
	This is the equation for spacelike $n$-dimensional hypersurfaces in $\R^{n,1} = \L^{n+1}$. When $n>1$, the equation is more complicated (see \cite[Appendix A]{LambertLotay21}), but the proof below is still valid. Recall that the $u_i$ are uniformly bounded in $C^{2,\alpha}$ on compact sets and converge to $0$ in $C^2$ on compact sets, and $v^2(\mcal{M}_i,\cdot)$ is also uniformly bounded on compact subsets of spacetime. Thus 
	\[ 
	\frac{D_ju_iD_ku_i}{1-\abs{Du_i}_{\R^n}^2} 
	\]
	converges to $0$ in $C^1$ (on compact sets) and $D_{jk}u_i$
	is bounded in $C^\alpha$ (on compact sets). It follows that $f_i$ converges to $0$ in $C^\alpha$ on compact sets. Thus the Schauder estimates for the heat equation imply that
	\[ 
	\norm{u_i|_K}_{2,\alpha} \leq c(K,n,m,\alpha) \left(\norm{u_i}_0 + \norm{f_i}_{0,\alpha} \right) \rightarrow 0 
	\]
	for every compact subset $K$ of $\Rnm \times \R$. But that contradicts the fact that $K_{2,\alpha}(\mathcal{M}_i, (0,0))$ was normalized to be $1$.  
\end{proof} 

Finally, we consider smooth, spacelike mean curvature flows with  boundary.  

\begin{lemma}\label{lemma:WhiteRegularityBoundary}
	Let $\mcal{M}$ be a smooth, graphical, spacelike mean curvature flow with boundary $\Gamma$ in $\RnmR$. Suppose the boundary has the form 
	\[
	\Gamma = L \times (-\infty,T]
	\] 
	for some $T \in (-\infty,\infty]$, where $L$ is an $(n-1)$-dimensional spacelike linear space of $\Rnm$. Suppose also that
	\[
    \Theta(\mcal{M},\Gamma,\infty) > 1 - \ol{\ve}/2 
    \]
    and 
    \[
    \sup_{X \in \mcal{M}} v^2(\mcal{M},X) \leq m + \delta,
    \]
	where $\delta \in (0,\ol{\delta})$ and $\ol{\ve}$ and $\ol{\delta}$ are as in Theorem \ref{th:LocalRegularity}. Then, after a suitable rotation, $\mcal{M}$ has the form 
	\[
	H \times \{0\}^m \times (-\infty,T],
	\]
	where $H = \{x\in\R^n : x_n \geq 0\}$.
\end{lemma}

\begin{proof}
    Since $x \in \Rnm \mapsto -x$ is an isometry of $\Rnm$ and $\Gamma = L \times (-\infty, T]$, where $L$ is an $(n-1)$-dimensional spacelike linear space of $\Rnm$, we have that
    \[
    	\mcal{M}' = \mcal{M} \cup \{(-x,t) : (x,t) \in \mcal{M}\}
    \]
    is a smooth, graphical spacelike mean curvature flow (without boundary). Moreover, 
	\[
	\Theta(\mcal{M}',\infty) = 2 \Theta(\mathcal{M}, \infty) = 2\left( \Theta(\mcal{M},\Gamma,\infty) - \frac{1}{2} \right) > 1 - \ol{\ve}.
	\]
	The results follows immediately from Theorem \ref{th:LocalRegularity}. 
\end{proof}

\begin{theorem}\label{th:LocalRegularityII}
	Let $\ve \in (0,\ol{\ve}/2)$ and $\delta \in (0,\ol{\delta})$, where $\ol{\ve}$ and $\ol{\delta}$ are as in Theorem \ref{th:LocalRegularity}. For every $0<\alpha<1$, $n$ and $m$, there exists a number $C=C(n,m,\alpha,\ve,\delta) < \infty$ with the following property. Suppose $\mcal{M}$ is a smooth, graphical, spacelike mean curvature flow with boundary $\Gamma$ in $P(X_0,R)$ such that 
	\begin{align*}
		\sup_{X \in \mcal{M} \cap P(X_0,r)} \Theta(\mcal{M},\Gamma, X,r) \geq 1- \varepsilon\\
	\end{align*}
	and
	\begin{align*}
		\sup_{X \in \mcal{M} \cap P(X_0,r)} v^2(\mcal{M},X) \leq m + \delta
	\end{align*}
	for some $r \in (0,R)$. Then
	\[ 
	K_{2,\alpha; P(X_0,r/2)}(\mcal{M}) \leq C \left( 1+K_{2,\alpha; P(X_0,r/2)}(\Gamma) \right). 
	\]
\end{theorem}

\begin{proof}
	If we assume that the theorem is false, we get a sequence of smooth, graphical, spacelike mean curvature flows $\mcal{M}_i$ with boundary $\Gamma_i$ in $P(X_i,R_i)$ such that 
	\begin{align*}
		\sup_{X \in \mcal{M}_i \cap P(X_i,r_i), } \Theta(\mcal{M}_i,\Gamma_i, X,r_i) \geq 1- \varepsilon,\\
		\sup_{X \in \mcal{M}_i \cap P(X_i,r_i)} v^2(\mcal{M}_i,X) \leq m + \delta,
	\end{align*}
	for some $r_i \in (0,R_i)$, and such that 
	\[
	\frac{K_{2,\alpha;P(X_i,r_i/2)}(\mcal{M}_i)}{1+K_{2,\alpha;P(X_i,r_i/2)}(\Gamma_i)} \to \infty. 
	\]
	Let $U_i = P(X_i,r_i/2)$. As before, by suitably translating and dilating, we may assume that $(0,0) \in \mcal{M}_i$ and that
	\begin{align*}
		K_{2,\alpha}(\mcal{M}_i,(0,0)) &= 1 \\
		d((0,0),U_i) &= d((0,0),U_i)K_{2,\alpha}(\mcal{M}_i,(0,0)) > \frac{1}{2}K_{2,\alpha;U_i}(\mcal{M}_i). 
	\end{align*}
	Thus $d((0,0),U_i) \to \infty$ and $ K_{2,\alpha;U_i}(\Gamma_i)/d((0,0),U_i) \to 0$, which implies $K_{2,\alpha}(\Gamma_i,\cdot) \to 0$ uniformly on compact subsets in spacetime. As before, the functions $K_{2,\alpha}(\mcal{M}_i,\cdot)$ are uniformly bounded on compact subsets in spacetime. Thus, by the Ascoli-Arzelà Theorem (c.f. \cite[Theorem 2.7]{White05}) and the uniform gradient bound, there exists a subsequence (which we denote by the same index) such that $\mcal{M}_i$ convergence in $C^2$ on compact subsets to a spacelike mean curvature flow $\mcal{M}$ with boundary $\Gamma$ in $\RnmR$. By the Ascoli-Arzelà Theorem (c.f. \cite[Theorem 2.7]{White05}) and the uniform gradient bound, $K_{2,\alpha}(\Gamma,\cdot) \equiv 0$ and so the boundary $\Gamma$ (after a suitable translation) must be one of the following:
	\[
	\Gamma = \begin{cases}
		\emptyset, \\
		L \times (-\infty,T],
	\end{cases}
	\]
	where $L$ is $(n-1)$-dimensional spacelike vector subspace of $\Rnm$ and $T \in [0,\infty]$. If $\Gamma = \emptyset$, the rest of the proof is just as in the case without boundary. Finally, if $\Gamma = L \times (-\infty,T]$, the hypothesis of Lemma \ref{lemma:WhiteRegularityBoundary} are satisfied, so (after a suitable rotation) $\mcal{M}$ has the form $H \times \{0\}^n \times (-\infty,T]$. The rest of the proof is exactly as in the case without boundary, except that now we now
	apply Schauder estimates at the boundary on a sequence of domains in $\RnmR$ converging to $H \times (-\infty,T]$.  
\end{proof}
	
\bibliographystyle{alpha}
\bibliography{References}
\end{document}